\documentclass{siamltex}

\usepackage[nosumlimits]{amsmath}
\usepackage{amsfonts,amssymb}
\usepackage{graphicx}
\graphicspath{{./Figures/}}
\usepackage[dvipsnames]{xcolor}
\usepackage{alltt}

\usepackage{tikz}
\usepackage{pgfplots}

\usepackage{ulem}
\usepackage{multirow}
\usepackage{multicol}
\usepackage[showonlyrefs]{mathtools}
\usepackage{cite} 
\usepackage[lined,boxruled,linesnumbered]{algorithm2e} 
\usepackage{subcaption}

\DeclareMathAlphabet{\mathbf}{OT1}{cmr}{bx}{it}
\newcommand{\dist}{\mathrm{d}}

\newcommand{\Rn}{\mathbb{R}^n}
\newcommand{\Rnn}{\mathbb{R}^{n \times n}}

\newcommand{\C}{\mathbb{C}}
\newcommand{\Cn}{\mathbb{C}^n}
\newcommand{\Cnn}{\mathbb{C}^{n \times n}}

\newcommand{\sizeL}{\ell}
\newcommand{\T}{\mathcal{T}}
\newcommand{\W}{\mathcal{W}}

\newcommand{\OO}{\mathcal{O}}
\newcommand{\Z}{\mathbb{Z}}

\DeclareMathOperator{\spec}{spec}
\DeclareMathOperator{\tr}{tr}

\DeclareMathOperator{\sign}{sign}

\DeclareMathOperator{\col}{col}

\DeclareMathOperator{\Li}{Li}
\DeclareMathOperator{\tridiag}{tridiag}
\DeclareMathOperator{\cov}{cov}

\newcommand{\udist}{\bar \dist}

\newtheorem{remarksimple}[theorem]{Remark}
\let\oldremarksimple\remarksimple
\renewcommand{\remarksimple}{\oldremarksimple\normalfont}
\newenvironment{remark}{\begin{remarksimple}}{\hfill$\diamond$\end{remarksimple}}
\newtheorem{examplesimple}[theorem]{Example}
\let\oldexamplesimple\examplesimple
\renewcommand{\examplesimple}{\oldexamplesimple\normalfont}
\newenvironment{example}{\begin{examplesimple}}{\hfill$\diamond$\end{examplesimple}}

\def\cvd{~\vbox{\hrule\hbox{%
   \vrule height1.3ex\hskip0.8ex\vrule}\hrule } }

\newcommand{\review}[1]{{\color{black}#1}}

\title{Analysis of probing techniques for sparse approximation and trace estimation of decaying matrix functions\thanks{This work was partially supported by Deutsche Forschungsgemeinschaft through the Collaborative Research Centre SFB-TRR55 ``Hadron Physics from Lattice QCD''.}}
\author{Andreas Frommer\thanks{Department of Mathematics, Bergische Universit\"at Wuppertal, 42097 Wuppertal, Germany, \texttt{\{frommer,schimmel\}@math.uni-wuppertal.de}} \and Claudia Schimmel$^\dag$\!\!\!\and Marcel Schweitzer\thanks{Mathematisch-Naturwissenschaftliche Fakult\"at, Heinrich-Heine-Universit\"at D\"usseldorf, Universit\"atsstra\ss{}e 1, 40225 D\"usseldorf, Germany. E-mail: \texttt{marcel.schweitzer@hhu.de.}}}
\date{\today}
\begin{document}
\renewcommand{\thefootnote}{\fnsymbol{footnote}}
\maketitle \pagestyle{myheadings} \thispagestyle{plain}
\markboth{A. FROMMER, C. SCHIMMEL AND M. SCHWEITZER}{ANALYSIS OF PROBING FOR DECAYING MATRIX FUNCTIONS}

\begin{abstract}
The computation of matrix functions $f(A)$, or related quantities like their trace, is an important but challenging task, in particular for large and sparse matrices $A$. In recent years, probing methods have become an often considered 
tool in this context, as they allow to replace the computation of $f(A)$ or $\tr(f(A))$ by the evaluation of (a small number of) quantities of the form $f(A)v$ or $v^Tf(A)v$, respectively. These \review{quantities can then efficiently be computed} by standard techniques like, e.g., Krylov subspace methods.
It is well-known that probing methods are particularly efficient when $f(A)$ is \emph{approximately sparse}, e.g., when the entries of $f(A)$ show a strong off-diagonal decay, but a rigorous error analysis is lacking so far. In this paper we develop new theoretical results on the existence of sparse approximations for $f(A)$ and error bounds for probing methods based on graph colorings. As a by-product, by carefully inspecting the proofs of these error bounds, we also gain new insights into when to stop the Krylov iteration used for approximating $f(A)v$ or $v^Tf(A)v$, thus allowing for a practically efficient implementation of the probing methods. 
\end{abstract}

 \begin{keywords}
matrix functions, sparse approximation, trace, decay bounds, graph coloring, probing method, Krylov subspace method 
 \end{keywords}

\begin{AMS}
05C12, 05C15, 15A16, 65F50, 65F60
\end{AMS}

\section{Introduction}\label{sec:introduction}
Matrix functions $f(A)$, where $f:\C \rightarrow \C$ is a scalar function and $A\in \Cnn$ is a square matrix, play an essential role in many areas of science and engineering. The inverse $A^{-1}$ is the most prominent example, another important case is the matrix exponential $f(A)=\exp(A)$, which is used for the numerical solution of time-dependent differential equations or the analysis of dynamical systems \cite{FrommerSimoncini2008}. 
For the computation of communicability measures in network analysis, the matrix exponential and the resolvent, generated by the scalar function $f(z)=(\alpha-z)^{-1}$ with $\alpha \in \C$ are widely used \cite{EstradaRodrigez2005,Estrada2011, EstradaHigham2010}.
The matrix sign function $f(A)=\sign(A)$ has applications in control theory \cite{Roberts1980,FrommerSimoncini2008} and lattice quantum chromodynamics \cite{Neuberger2000,Block2007, Eshof2002}. Inverse fractional powers $f(A)=A^{-\alpha}$ with $\alpha \in (0,1)$ are strongly related to the matrix sign function and arise in generalized eigenvalue problems \cite[Section 15.10]{Parlett1980}, fractional differential equations \cite{Burrage2012} or sampling from multivariate Gaussian distributions~\cite{PleissEtAl2020}.

For many of these applications, the explicit computation of $f(A)$ is not feasible as the matrix $A$ is typically large and sparse, while $f(A)$ is \review{generally} a dense matrix. Therefore, one has to resort to approximation techniques when $f(A)$ or a related quantity like $f(A)b$, $b \in \Cn$, the diagonal $\diag(f(A))$ or the trace $\tr(f(A))$ is required. This work focuses on sparse approximations for the whole matrix $f(A)$ on the one hand and on approximating $\tr(f(A))$ on the other hand.  Computing the trace $\tr(f(A))$ is a relevant task.  For example, the trace of the inverse is required in the study of fractals \cite{Sapoval1991}, generalized cross-validation and its applications \cite{GolubHeathWahba1979,GolubMatt1995}, or when computing disconnected fermion loop contributions in lattice quantum chromodynamics (QCD) \cite{SextonWeingarten1994,DongLiu1993}. In network analysis, the Estrada index---a total centrality measure for networks---is defined as the trace of the exponential of the adjacency matrix of a graph \cite{EstradaHigham2010,Ginosar2008} and an analogous measure is given by the trace of the resolvent \cite[Section 8.1]{Estrada2011}. For Hermitian positive definite matrices $A$, one can compute the log-determinant $\log(\det(A))$ as the trace of the logarithm of $A$. Amongst others, the log-determinant is needed in machine learning and related fields \cite{Rasmussen2005,Rue2005}. \review{Further applications are discussed in \cite{meyer2021hutch++,SaadUbaru2017,SaadUbaruJie2017}}.

In recent years, probing methods~\cite{BollhoeferEftekhariScheideggerSchenk2019,LaeuchliStathopoulos2020,Stathopoulos2013,TangSaad2012} have emerged as an important tool \review{for computing sparse approximations of matrix functions or estimating their trace}. They obtain approximations by evaluating a small number of matrix-vector products or bilinear forms involving $f(A)$, which can be done by standard techniques (e.g., Krylov subspace methods). We brief\/ly summarize the main idea of these methods in the following.

\subsection{Probing methods}
Recall that the (directed) graph $G(A) = (V,E)$  of a sparse matrix $A \in \Cnn$ is given by the vertices $V = \{1,\dots,n\}$ and egdes $E = \{(i,j): a_{ij} \neq 0, i\neq j\}$. By $\dist(i,j)$ we denote the geodesic distance, i.e., the length of the shortest path, from node $i$ to node $j$ in $G(A)$ and by $\udist(i,j)$ the geodesic distance in the corresponding undirected graph $|G(A)|$ which results from $G(A)$ by removing the direction of the edges. 

Given a partitioning $\mathcal{V} = \{V_1,\ldots,V_m\}$ 
of the nodes $V$ of $G(A)$, i.e.,
\begin{equation}\label{eq:node_partitioning}
V=V_1 \cup \ldots \cup V_m, \enspace V_\ell \neq \emptyset \mbox{ for } \ell = 1,\ldots,m  \text{ and } V_\ell \cap V_p = \emptyset \text{ for } \ell \neq p,
\end{equation}
the corresponding probing vectors are defined as
\begin{equation}\label{eq:probing_vectors}
v_\ell:=\sum\limits_{i\in V_\ell} e_i, \; \ell \in \{1,\ldots,m\} \quad  \mbox{($e_i$ is the $i$th canonical unit vector).}
\end{equation}
The vectors $v_\ell$ can be used to, e.g., 
estimate $\tr(f(A))$ via
\begin{equation}\label{eq:probing_trace_approximation}
\tr(f(A)) \approx \T(f(A)) := \sum\limits_{\ell=1}^m v_\ell^H f(A) v_\ell,
\end{equation} 
or even construct a sparse approximation to $f(A)$ itself via
\begin{equation}\label{eq:probing_sparse_approximation}
 [f(A)^{[d]}]_{ij}:= \begin{cases}
                            [f(A)v_\ell]_i \text{ for } j \in V_\ell &\text{ if } \udist(i,j) \leq d, \\
                            0 &\text{ if } \udist(i,j) > d,
                           \end{cases}
\end{equation}
where $d$ is a prescribed distance threshold. We refer to, e.g.,~\cite{BollhoeferEftekhariScheideggerSchenk2019,LaeuchliStathopoulos2020,Stathopoulos2013,TangSaad2012} for detailed discussions of such probing approaches and just expose the main motivation: In the (unrealistic) situation that $f(A)$ is a sparse matrix with 
$[f(A)]_{ij} = 0$ for $\dist(i,j) > d$, if the sets $V_\ell$ are chosen such that $[f(A)]_{ij} = 0$ for $i,j \in V_\ell, i\neq j$, both approximations~\eqref{eq:probing_trace_approximation} and~\eqref{eq:probing_sparse_approximation} are actually exact. Therefore, if $f(A)$ is \emph{approximately sparse}, and the $V_\ell$ are built such that $[f(A)]_{ij}$ is small for $i,j \in V_\ell, i\neq j$, we can expect probing methods to yield accurate approximations.


\subsection{Exponential decay in matrix functions}
To make the notion of $f(A)$ precise, recall that $f(A)$ is defined if for all eigenvalues $\lambda$ of $A$ all derivatives of $f$ at $\lambda$  up to order $\nu(\lambda)-1$ exist, where $\nu(\lambda)$ is the multiplicity of the elementary factor $(z-\lambda)$ in the minimal polynomial of $A$, see~\cite{Higham2008}. We tacitly assume that this is always fulfilled whenever we consider $f(A)$. Note that $f(A)$ is then given as the polynomial in $A$ which interpolates $f$ on the spectrum of $A$ in the Hermite sense.

One special form of approximate sparsity in $f(A)$ that is frequently encountered in practice is \textit{exponential decay} of the entries of $f(A)$ away from the sparsity pattern of $A$. 
\begin{definition}\label{def:exp_decay}
 Let $A \in \Cnn$ and let $f$ be defined on the spectrum of $A$. The matrix $f(A)$ has exponential decay (away from the sparsity pattern of $A$) if 
\begin{equation}\label{eq:generalDecaybound}
  |[f(A)]_{ij}| \leq C  q^{\dist(i,j)} \text{ for } i,j \in \{1,\ldots,n\},
  \end{equation}
 where \review{$C > 0, 0 \leq q < 1$ are constants independent of $i,j$ and} $\dist$ is the geodesic distance in $G(A)$. 
\end{definition}
%

\review{To highlight the importance of the geodesic distance in $G(A)$ in Definition~\ref{def:exp_decay}, we also say that $f(A)$ has exponential decay \emph{with respect to $G(A)$.}}

Of course, for given $A$ the relation \eqref{eq:generalDecaybound} can always be satisfied if we choose $C$ and $q \in(0,1)$ large enough. To be meaningful, the concept of exponential decay therefore implicitly assumes that $C$ and, in particular, $q$ are not too large or that \eqref{eq:generalDecaybound} holds uniformly for a whole, possibly infinite, family of matrices. \review {For example,  exponential decay is a meaningful concept if the family of matrices is such that, on average, the distance $\dist(i,j)$ of two nodes increases as the size of the matrices increases. This is typically the case for matrices arising from refinements of a discretization.  It is not the case, for example, if the matrices are the adjacency matrices of a family of small world graphs \cite{Estrada2011} as they arise in network modeling.}

%
%
%

Decay in matrix functions has been studied extensively, starting with~\cite{DemkoMossSmith1984}, where accurate exponential decay bounds were presented for inverses of banded (Hermitian positive definite) matrices. Lots of other results and decay bounds for different types of functions and matrices can be found, e.g., in \cite{BenziGolub1999, BenziBoito2014, EijkhoutPolman1988, PT18, Nabben1999, BenziSimoncini2015, BenziRazouk2007, FrommerSchimmelSchweitzer2018_1,FrommerSchimmelSchweitzer2018_2,PozzaSimoncini2016}. \review{Depending on the properties of $f$, even superexponential decay might occur~\cite{BenziBoito2014, PozzaSimoncini2016, benzi20, BenziSimoncini2015}, a phenomenon most frequently encountered for entire functions. 
Many known decay results are derived by exploiting} properties of polynomial approximations to  $f$. Indeed, if $i$ and $j$ have distance $\dist(i,j)$ in the graph $G(A)$, then for every polynomial $p_s$ of degree at most $s=\dist(i,j)-1$ we have $[p_s(A)]_{ij} = 0$, see \cite{BenziRazouk2007} which implies
$$|[f(A)]_{ij}|=|[f(A)]_{ij}-[p_s(A)]_{ij}|\leq \|f(A) - p_s(A) \|_2.  
$$
Herein, $\|f(A)-p_s(A)\|_2$ can be bounded further due to the following important approximation result which uses the numerical range $\mathcal{W}(A) = \{x^HAx: \|x\| = 1\}$.

\begin{theorem}  \label{the:crouzeix}
Let $A\in \Cnn$ and let \review{$g:\mathcal{W}(A) \to \C$ be defined on the numerical range of $A$}. Then
\begin{equation}\label{eq:crouzeix}
\|g(A) \|_2 \leq K \max_{z\in \review{\mathcal{W}(A)}} |g(z)|,
\end{equation}
where $K = 1$ if $A$ is normal and $K = 1 +\sqrt{2}$ otherwise.
\end{theorem}

Note that this result is almost a triviality for $A$ Hermitian, while the general case is much more involved, see~\cite{CrouzeixPalencia2017}.  
Applying Theorem~\ref{the:crouzeix} to $g = f -p_s$ immediately gives the following result which relates the accuracy of polynomial approximation to exponential decay in the matrix function. 
We will use it several times in this paper. 

\begin{theorem} \label{the:polynomial_approx_decay} 
 \review{Let $\mathcal{W}$ be a compact set and assume that}
 \begin{equation}\label{eq:polynomial_approximation_decay}
\min\limits_{p_s \in \Pi_s} \max_{z \in \mathcal{W}} |f(z)-p_s(z)| \leq Cq^s,
\end{equation}
where \review{$C > 0, 0 \leq q < 1$} and $\Pi_s$ is the set of all polynomials with degree $\leq s$.
Then, if $\W(A) \subseteq \review{\mathcal{W}}$ we have 
$$|[f(A)]_{ij}| \leq \|f(A) - p_s(A) \|_2 \leq KCq^s \mbox{ whenever } \dist(i,j) >  s.
$$
\end{theorem}

%
Thus, uniform exponential decay bounds for a family of matrices can be obtained if there is a common superset $\review{\mathcal{W}}$ of their numerical ranges for which \eqref{eq:polynomial_approximation_decay} holds, as it is, e.g., the case for the results in~\cite{DemkoMossSmith1984, BenziRazouk2007, FrommerSchimmelSchweitzer2018_1}.

In our error analysis to come we will sometimes distinguish between general exponential decay bounds for $f(A)$ and bounds which are explicitly based on \eqref{eq:polynomial_approximation_decay}.


\subsection{Outline of the paper}
The main goal of this paper is to obtain guidelines for choosing the sets $V_\ell$ in~\eqref{eq:node_partitioning} and using this information to derive rigorous error bounds for the resulting approximations \eqref{eq:probing_trace_approximation} and \eqref{eq:probing_sparse_approximation} in case that $f(A)$ exhibits an exponential decay property. In addition, our analysis also sheds light onto when to stop Krylov subspace iterations used for approximating $f(A)v_\ell$ or $v_\ell^H f(A)v_\ell$, respectively, in order to reach an implementation that is as efficient as possible without sacrificing accuracy.

This paper is organized as follows. In Section~\ref{sec:coloring} we discuss the distance-$d$ graph coloring problem as it forms the basis of the discussed probing methods. In Section~\ref{sec:sparseapprox} we first give some new theoretical results on the existence of sparse approximations of matrix functions and then prove new error bounds for the approximation~\eqref{eq:probing_sparse_approximation} for $f(A)$.  Section~\ref{sec:trace} covers error bounds for the approximation~\eqref{eq:probing_trace_approximation} of $\tr(f(A))$, while Section~\ref{sec:krylov} uses the insights from Section~\ref{sec:sparseapprox} to develop stopping criteria for the Krylov subspace approximation inside the probing method. We illustrate the quality of the derived bounds in numerical experiments reported in Section~\ref{sec:experiments}. Concluding remarks are given in Section~\ref{sec:conclusion}.

\section{Distance-$d$ coloring}\label{sec:coloring}
 
The quality of the approximations~\eqref{eq:probing_trace_approximation} and~\eqref{eq:probing_sparse_approximation} crucially depends on the partitioning~\eqref{eq:node_partitioning}. If $f(A)$ has exponential decay
with respect to $G(A)$, good partitionings can be obtained via graph colorings.

 \begin{definition}
 A distance-$d$ coloring of a graph $G=(V,E)$  is a mapping $\col: V \rightarrow \{1,\ldots,m\}$ such that $\col(i)\neq \col(j)$ if $\dist(i,j)\leq d$. A distance-$d$ coloring is optimal if the number $m$ of colors is minimal among all distance-$d$ colorings of $G$. 
 \end{definition}

For $d=1$, the computation of an optimal distance-$d$ coloring corresponds to the classical graph coloring problem, which is known to be NP complete for general graphs \cite{Kubale2004}. 
In our setting, we are mainly interested in low-cost methods for computing a distance-$d$ coloring with a sufficiently small number of colors.
Efficient ways for computing such colorings of graphs are usually based on greedy strategies, see, e.g., \cite{Kubale2004}. For example, a distance-$d$ coloring of a graph $G$ with $V=\{w_1,\ldots,w_n\}$ can be obtained via  $\col(w_1)=1$ and
$\col(w_i) = \min\{k>0: k \neq \col(w) \text{ for all } w\in W_i\}$ for $i=2,\ldots,n$
where
\begin{equation} \label{eq:Wi_def}
W_i:=\{w \in \{w_1,\ldots,w_{i-1}\} : \dist(w_i,w)\leq d\}.
\end{equation}
This coloring uses at most $\Delta(G)^d+1$ colors and can be implemented with cost $\OO(\Delta(G)^d n)$, where $\Delta(G)$ is the maximal degree of $G$~\cite[Proposition 4.2]{Schimmel2019}. In the next two sections, we discuss special classes of graphs where a (not necessarily optimal) distance-$d$ coloring can be obtained with cost $\OO(n)$.

\subsection{Distance-$d$ colorings for graphs of banded matrices}\label{subsec:coloring_banded}
Let $A$ be a banded matrix with semi-bandwidth $\beta$, i.e.\ $[A]_{ij} = 0$ whenever $|i-j| >  \beta$. 
%
Then it is easy to verify that a distance-$d$ coloring for $G(A)$ with $m= d\beta+1$ colors is given by 
\begin{equation}\label{eq:coloring_banded}
\col(i)= (i-1) \bmod (d\beta +1) +1, \quad i=1,\ldots,n,
\end{equation}
and this coloring is optimal if all entries within the band of $A$ are \review{nonzero}.
If $A$ is sparse but not banded with small $\beta$, 
one can first determine an ordering of the nodes which aims at obtaining a (relatively) small bandwidth for the correspondingly permuted matrix and then define the coloring via \eqref{eq:coloring_banded} on the permuted nodes. Finding a permutation resulting in a small bandwidth is an important topic in the context of direct solvers for linear systems, and lots of low-cost methods have been proposed over the years; see, e.g., \cite{CuthillMcKee1969, Smyth1985, Cheng1973,ReidScott2006, GibbsPooleStockmeyer1967, LimRodiguesXiao2007}. A heuristic based on level sets is at the basis of the classical Cuthill-McKee algorithm \cite{CuthillMcKee1969} with cost $\OO(|V|+|E|)$, and we refer to \cite{GonzagadeOliveira2018} for an overview and comparison of various other recent low-cost heuristics. 
The cost for computing the coloring is dominated by the cost for the computation of the permutation of the nodes.

\subsection{Distance-$d$ colorings for regular lattices}\label{subsec:coloring_lattice}
As another special case, assume that the graph $G=(V,E)$ is a regular $D$-dimensional lattice for $D>1$.  For $D=1$, the adjacency matrix is tridiagonal, a situation already covered by the banded case discussed before.

First, note that the greedy coloring approach can be made more explicit when applied to regular lattices: Each node $w$ in a regular $D$-dimensional lattice can be identified with its coordinates $w=(w^{[1]},\ldots,w^{[D]})\in \Z^D$, see Figure~\ref{fig:graph_lattice}(left) for an illustration. Using this representation, we have $d(v,w)=\|v-w\|_1= |v^{[1]}-w^{[1]}|+ \cdots + |v^{[D]}-w^{[D]}|$ and thus $W_i$ from \eqref{eq:Wi_def} is given as
$$W_i=\{w\in \{w_1,\ldots,w_{i-1}\}: \|w-w_i\|_1\leq d\}.$$
For an infinite lattice it is known  \cite[Theorem 2.7]{BeckRobins2015} that the cardinality $\sizeL_D(d)$ of the set $\{z\in \Z^D: \|z\|_1 \leq d\}$ 
is given as 
\begin{equation} \label{eq:size_of_distance_set_in_grid}
\sizeL_D(d)=\sum\limits_{k=0}^D \binom {D}{k} \binom{d+D-k}{D} \quad \left(\mbox{where $\binom{d+D-k}{k}=0$ if $d < k$}\right).
\end{equation}
So, in a greedy algorithm, $W_i$ can be obtained by examining at most $\sizeL_D(d)-1$ nodes and check whether they have already been colored.
Alternatively, a distance-$d$ coloring for regular $D$-dimensional lattices can also be obtained directly, due to the following result which we prove in Appendix~\ref{sec:appendix}.

\begin{figure}
    \centering
    \begin{minipage}[t]{0.47\linewidth}
        \centering
        \includegraphics[width=.75\linewidth]{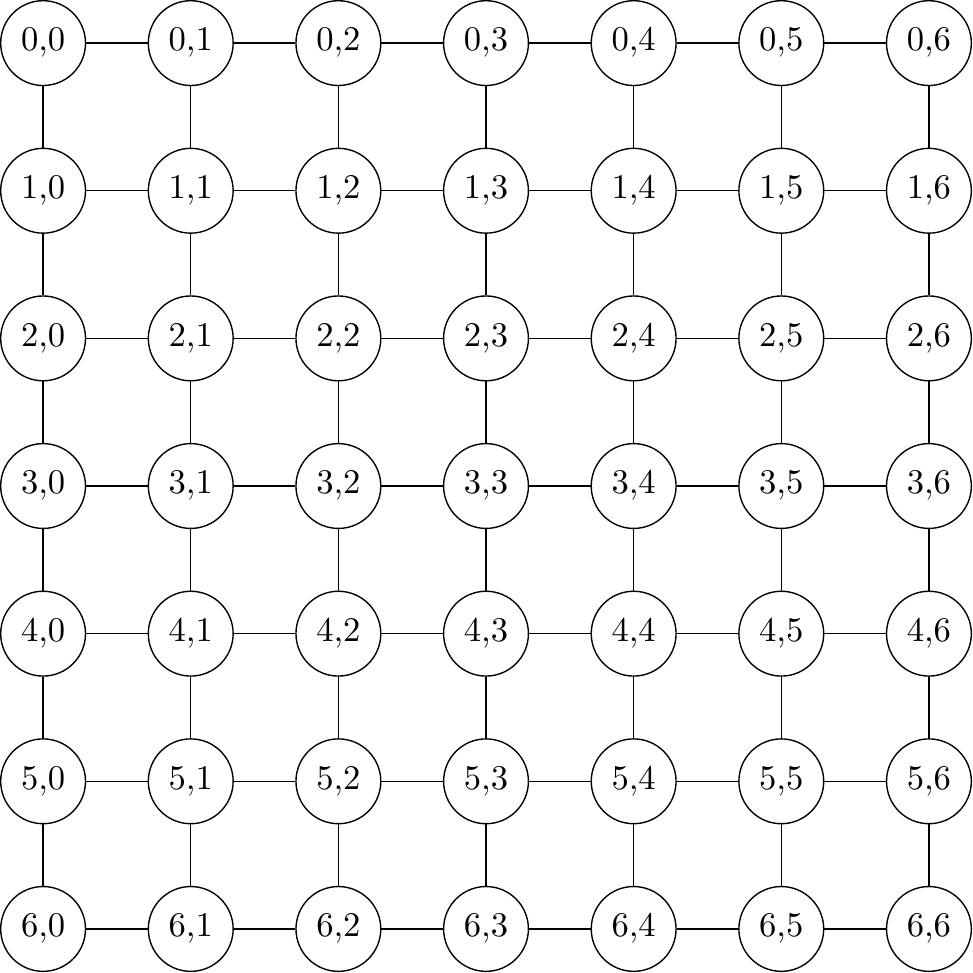}
        \caption{Two-dimesional $7\times7$ lattice, where each node is defined by two coordinates $0\leq w_1,w_2 \leq 6$. }
        \label{fig:graph_lattice}
    \end{minipage}
    \hfill
    \begin{minipage}[t]{0.47\linewidth}
        \centering
        \includegraphics[width=.75\linewidth]{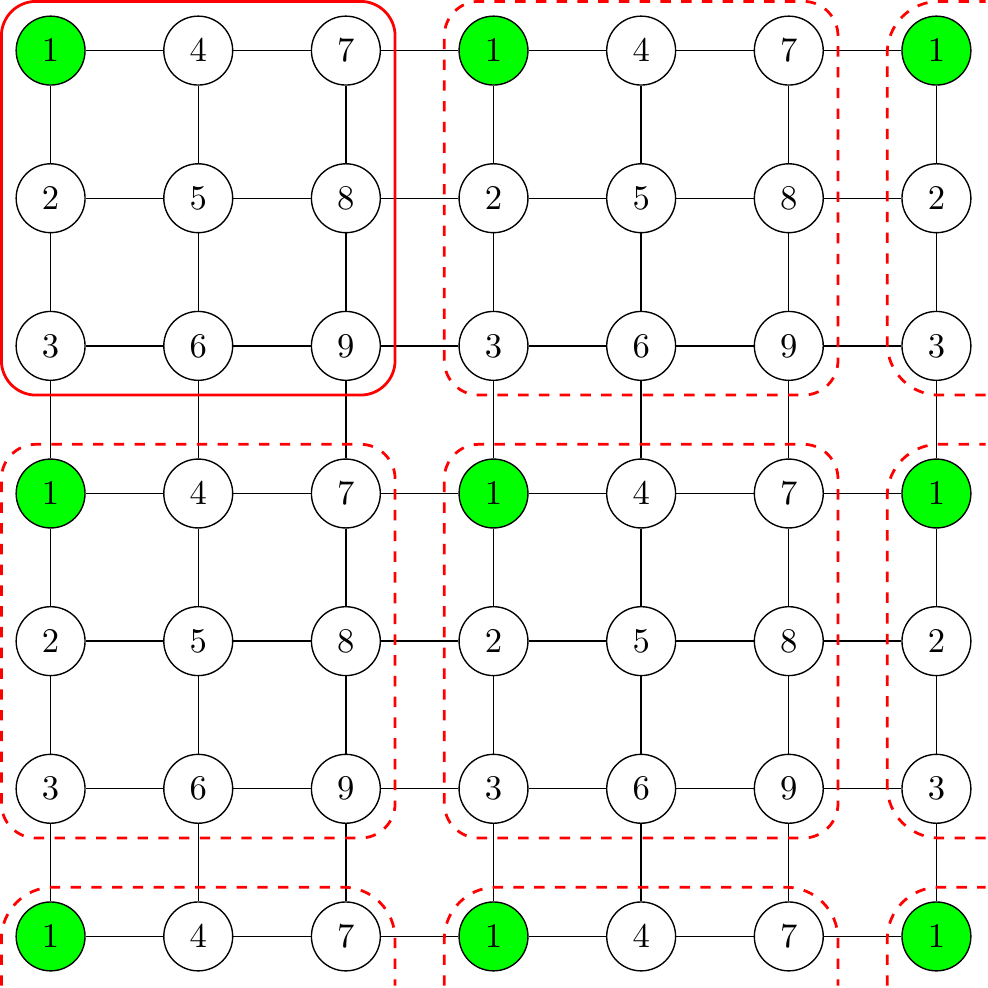}
        \caption{Distance-2 coloring produced by Theorem \ref{the:lattice_col}.}
        \label{fig:graph_lattice_col}
    \end{minipage}%
\end{figure}


\begin{theorem}\label{the:lattice_col}
 Let $G=(V,E)$ be a $D$-dimensional $N_1 \times N_2 \review{\times} \cdots \times N_D$ lattice. Let any node $w\in V$ be defined by its coordinates $w=(w^{[1]},\ldots,w^{[D]})$, with $0 \leq w^{[i]} \leq N_i-1$, $i\in 1,\ldots, D$. Then a distance-$d$ coloring with $(d+1)^D$ colors is given by
 \begin{equation}\label{eq:coloring_lattice}
 \col(w)= \left (\sum\limits_{k=0}^{D-1} \widetilde{w^{[k]}} (d+1)^k \right ) + 1,
 \quad \text{where} \quad
  \widetilde{w^{[k]}} = w^{[k]} \bmod (d+1).
 \end{equation}
\end{theorem}

Let us note that for $D=2$, an optimal distance-$d$ coloring is explicitly known with $\left \lceil  \frac 12(d+1)^2 \right \rceil$ colors; see~\cite{FertinGodardRaspaud2003}, while the coloring given in Theorem~\ref{the:lattice_col} needs approximately twice as many colors. 

Two characteristics of the coloring of Theorem~\ref{the:lattice_col} for general $D$ will further be exploited in the error analysis presented in Sections~\ref{sec:sparseapprox} and~\ref{sec:trace}: Firstly, the construction is based on the fact that we color all nodes $w$ in the cube $\{w: 0 \leq w^{[k]} \leq d \mbox{ for } k=1,\ldots,D\}$ with $(d+1)^D$ colors as illustrated in Figure~\ref{fig:graph_lattice_col} (red, solid \review{frame}). This coloring is then repeated by shifting this initial cube through the entire lattice (red, dashed \review{frames}).
Secondly, with this coloring every color class can be interpreted as representing a coarse grid, where the distances between the nodes in one color class are multiples of $d+1$. This is illustrated in Figure~\ref{fig:graph_lattice_col} where the green, \review{(filled)} nodes represent one color class. 
\begin{remark}\label{rem:hierarchical_coloring}
For $D$-dimensional lattices with an equal number of nodes in each dimension, a recursively computable \emph{hierarchical distance-$d$ coloring} 
was introduced in \cite{Stathopoulos2013} for distances $d=2^i, \, i=0,1,\ldots$, using $2^{Di+1}=2d^D$ colors. This approach was recently extended to lattices with an uneven number of nodes per dimension and even more general graphs in~\cite{LaeuchliStathopoulos2020}. Note that for $D$ small and $d$ not too small, we have $(d+1)^D < 2d^D$. For example, $(d+1)^2 < 2d^2$ as soon as $d > 2$ and $(d+1)^3 < 2d^3$ as soon as $d >3$.
For the analysis in Section~\ref{sec:sparseapprox} and~\ref{sec:trace}, the colorings discussed in the present paper are more appropriate and the analysis of the hierarchical probing approach is beyond the scope of this work.
\end{remark}

 We end the discussion of regular lattices with the following result which bounds the number of nodes that have exact distance $d$ from a given node. The rather technical, combinatorial proof is presented in Appendix~\ref{sec:appendix}.

\begin{lemma}\label{lemma:levelsets}
 Let $\sizeL_D^=(d) := |\{z\in \Z^D: \|z\|_1 = d\}|,$ then $\sizeL_D^=(d)\leq 2Dd^{D-1}$.
\end{lemma}

\section{Sparse approximation of matrix functions}\label{sec:sparseapprox}

In this section we analyze the error of the approximation~\eqref{eq:probing_sparse_approximation} when one of the colorings from Section~\ref{sec:coloring} is used. Before doing so, we first discuss some general results on the existence and quality of sparse approximations to reveal what is achievable at all.   

\subsection{General results on sparse approximations} 
\label{subsec:sparse_approx_theory}
We place ourselves in a slightly broader context, as it was also done in \cite{BenziBoitoRazouk2013}, and formulate sparse approximation results in terms of a matrix $B \in \Cnn$ (instead of $f(A)$) with a decay property with respect to a general graph $G=(V,E)$ with $V=\{1,\ldots,n\}$ (instead of $G(A))$).
The following essential result from \cite{BenziBoitoRazouk2013} forms the basis for sparse approximations of matrices with exponential decay.

\begin{theorem}\label{the:Benzi2013}
Let $\{B_s\}_{s \in \mathcal{S}}$ be a family of $n_s\times n_s$ matrices having exponential decay with respect to a family of corresponding graphs $\{G_s\}_{s \in \mathcal{S}}$ with geodesic distances $\dist_s$,
\begin{equation*} \label{eq:general_decay}
|[B_s]_{ij}| \leq Cq^{\dist_s(i,j)}, \enspace i,j=1\ldots,n,  
\end{equation*}
with $C> 0, q \in (0,1)$ independent of $s$. 
Assume that the graphs have bounded maximal degree $\Delta(G_s)\leq c$ for all $s$. Then for every $\varepsilon>0$, $B_s$ contains at most $\OO(n_s)$ entries greater than $\varepsilon$ in magnitude.
\end{theorem}

Furthermore, the following result for matrices with exponential {\em off-diagonal} decay was also given in \cite{BenziBoitoRazouk2013}.

\begin{theorem}\label{the:sparseApprox_banded}
 Let $\{B_s\}_{s \in \mathcal{S}}$  be a family of $n_s \times n_s$ matrices with 
 $$|[B_s]_{ij}| \leq C q^{|i-j|}, \enspace i,j=1\ldots,n, 
 $$
with $C> 0, q \in (0,1)$ independent of $s$. Then for $\varepsilon >0$ there exists $\widetilde{m}$ independent of $s$ such that 
 $$ \|B_s - B_s^{(m)}  \|_1 < \varepsilon \enspace \mbox{for $m>\widetilde{m}$},
 $$
 where $B_s^{(m)} \in \mathbb{C}^{n_s \times n_s}$ is the banded matrix with  $[B_s^{(m)}]_{ij} = [B_s]_{ij}$ for $|i-j|\leq m$ and $[B_s^{(m)}]_{ij}=0$ for $|i-j|>m$. In particular, for any fixed $m > \widetilde{m}$ the matrices $B_s^{(m)}$ contain $\OO(n_s)$ nonzeros. 
\end{theorem}

To obtain a generalization for matrices with general exponential (not necessarily off-diagonal) decay, we define the level sets of a node $j\in V$ in a graph $G = (V,E)$ with $|V| = n$ as 
\begin{eqnarray*}
L^{(\delta)}(j) &:=& \{ i \in V, \, \dist(i,j) = \delta \}, \quad \delta = 0,\ldots,n-1, \\
L^{(\infty)}(j) &:=& \{ i \in V, \, \dist(i,j) = \infty \}.
\end{eqnarray*}

Note that for any node $j$ we have
$$V \, = \cup_{\delta=0}^{n-1} L^{(\delta)}(j)  \cup L^{(\infty)}(j),$$
since every node $i$ has either 
distance smaller than $n$ 
from $j$ or cannot be reached from $j$, in which case $i \in  L^{(\infty)}(j)$. With these notations we can give the following generalization of Theorem~\ref{the:sparseApprox_banded}.
\begin{theorem}\label{the:sparseApprox_general}
 Let $\{B_s\}_{s \in \mathcal{S}}$ be a family of $n_s\times n_s$ matrices having exponential decay 
 $$|[B_s]_{ij}|\leq C q^{\dist_s(i,j)}, \enspace i,j = 1,\ldots,n$$
 with respect to the distances $\dist_s$ in a family of graphs $\{G_s\}_{s \in \mathcal{S }}$, where $C>0, q\in(0,1)$ are independent of $s$. Furthermore, assume that for all nodes $j$  all level sets $L^{(\delta)}_s(j)$ of all graphs $G_s$ are polynomially bounded, i.e., we have
\begin{equation}\label{eq:polynomially_bounded_level_sets} 
|L^{(\delta)}_s(j)| \leq K \, \delta^\alpha
\end{equation}
with $K>0$ and $\alpha >0$, both independent of $s$ and $j$. For $m>0$ define the matrix $B_s^{(m)}$ via
\begin{equation*}\label{eq:sparse_approx_exact}
[B_s^{(m)}]_{ij} = \begin{cases}
                      [B_s]_{ij}  &\mbox{ if } \dist_s(i,j) \leq m \\
                      0  &\mbox{ otherwise.} 
                     \end{cases}
\end{equation*}
Then for $\varepsilon>0$ there exists $\widetilde{m}$ independent of $s$ such that $\|B_s-B_s^{(m)}\|_1 < \varepsilon$
for all $m > \widetilde{m}$. Moreover, for any fixed $m > \widetilde{m}$ the matrices $B_s^{(m)}$ contain $\OO(n_s)$ nonzeros.

\end{theorem}
\begin{proof}
Let $m_1=m_1(q,\alpha)$ be such that $\delta^\alpha q^{\frac \delta 2} <1$ holds for $\delta > m_1$. Then for $m > m_1$ we obtain

\begin{eqnarray*}
\|B_s - B_s^{(m)}\|_1 &=& \max_{j=1}^{n_s} \sum\limits_{i=1}^{n_s} \left | [B_s]_{ij} - [B_s^{(m)}]_{ij} \right | \, = \, \max_{j=1}^{n_s} \sum\limits_{\substack{i\\\dist_s(i,j)>m}}\left |[B_s]_{ij} \right | \\
&\leq& \max_{j=1}^{n_s} \! \sum\limits_{\substack{i\\ \dist_s(i,j)>m}} \! Cq^{\dist_s(i,j)} \, = \, \max_{j=1}^{n_s} \; C \!\sum\limits_{\delta=m+1}^{n_s-1} | L^{(\delta)}_s(j) | q^\delta \, \leq \,  CK \! \sum\limits_{\delta=m+1}^{n_s-1} \; \delta^\alpha q^\delta  \\ 
& =& CK \! \sum\limits_{\delta=m+1}^{\infty} \delta^\alpha q^{\frac \delta 2} q^{\frac \delta 2}  \, \leq \, CK \! \sum\limits_{\delta=m+1}^{\infty} q^{\frac \delta 2}  \, \leq \,  CK  \; \frac{\sqrt{q}^{ m+1}}{1-\sqrt{q}} \, .
\end{eqnarray*}

Let $m_2=m_2(q,\varepsilon)$ be such that 
$$CK  \; \frac{\sqrt{q}^{{m+1}}}{1-\sqrt{q}}<\varepsilon$$
holds for $m > m_2$. Then for $m > \widetilde{m} :=\max\{m_1,m_2\}$ we have $\|B_s-B_s^{(m)}\|_1 < \varepsilon$, and the number of \review{nonzero} elements in $B_s^{(m)}$ is at most 
$$
\sum_{j=1}^{n_s} \sum_{\delta=0}^m |L^{(\delta)}_s(j)| \leq  n_s \left( 1+ \sum_{\delta = 1}^m K\delta^{\alpha}\right ) = \OO(n_s),
$$
from which the assertion of the theorem follows.
\end{proof}

Note that off-diagonal decay is equivalent to decay with respect to a chain graph and thus Theorem~\ref{the:sparseApprox_banded} is covered by this theorem: For a chain the level sets $L^{\delta}(j)$ contain at most two elements, i.e., we have $\alpha = 0$. For general $\alpha>0$ we now have a similar result for other important cases, e.g., when the graphs $G_s$ are regular $D$-dimensional lattices. 

Theorem~\ref{the:sparseApprox_general} was formulated in \cite{BenziBoitoRazouk2013} with the assumption \eqref{eq:polynomially_bounded_level_sets} on polynomially bounded level sets replaced by the less restrictive assumption that the family of graphs $\{G_s\}$ has bounded maximal degree. This turns out to have been too optimistic, as the following example shows.

\begin{example}
 Let $0< q < 1$, let $t \in \mathbb{N}$ be such that $tq > 1$ holds, and let $G_p$ be the full $t$-ary tree with height $p$, which has $n_p = 1 + t + \cdots + t^{p} = (t^{p+1}-1)/(t-1)$ nodes. Then the maximal degree of the graph $G_p$ is bounded, $\Delta(G_p) = t+1$. 
Let $j$ be the root of this tree so that the level set $L^{(\delta)}_p(j)$ is formed exactly by all nodes at depth $\delta$ in the tree, implying
\[
|L^{(\delta)}_p(j)| = t^{\delta},  \delta = 0,\ldots,p, \enspace  L^{(\infty)}_p(j) = \emptyset.
\]
Let $B_p$ be the $n_p \times n_p$ matrix with $[B_p]_{ij} = q^{\dist_p(i,j)}$, where $\dist_p$ is the distance in $G_p$. Then $B_p$ has exponential decay with respect to $G_p$, and for all $m$ we have 
$$\| B_p - B_p^{(m)}\|_1 \geq \! \sum\limits_{\dist_p(i,j) > m } | [B_p^{(m)}]_{ij} | = \! \sum\limits_{\delta = m+1}^p | L^{(\delta)}_p(j) | q^\delta = \! \sum\limits_{\delta = m+1}^p  t^\delta q^\delta \geq (p-m)(tq)^{m+1},$$
where the last inequality holds because of $tq > 1$. Thus, the first $m$ for which $\| B_p - B_p^{(m)}\|_1 < 1$ holds is $m=p = \Omega(\log n)$, in which case we have $B_p^{(m)} = B_p$.
\end{example}

In this example the exponential decay in $B_p$ is not enough to compensate the exponential growth of the level sets. 
This motivated  condition~\eqref{eq:polynomially_bounded_level_sets} in Theorem~\ref{the:sparseApprox_general}.

\subsection{Analysis of probing for sparse approximation of $f(A)$}\label{subsec:sparse_approx_probing}
We now turn back to the specific situation where $B = f(A)$ and $G = G(A)$.
The existence results of the previous section do not reveal how a sparse approximation is obtained in practice without computing $f(A)$. We now investigate the probing approximation \eqref{eq:probing_sparse_approximation} for obtaining such an approximation.

The following result gives an entrywise bound for the probing approximation $f(A)^{[d]}$ from \eqref{eq:probing_sparse_approximation} provided the probing vectors are obtained from a distance-$2d$ coloring of $|G(A)|$.

\begin{proposition}\label{prop:errorboundentry}
Let $f(A)$ have exponential decay \eqref{eq:generalDecaybound}, let the sets $V_\ell$ be the color classes of a distance-2d coloring of $|G(A)|$ and $v_{\ell}$ the corresponding probing vectors \eqref{eq:probing_vectors}. Let $f(A)^{[d]}$ be the approximation defined by \eqref{eq:probing_sparse_approximation}.
Then with $\varepsilon = Cq^d$  the following entrywise error bound holds for all $i,j \in \{1,\ldots,n\}$
\begin{equation*}\label{eq:sparseApprox_error_entrywise}
|[f(A)]_{ij} - [f(A)^{[d]}]_{ij}| \leq \begin{cases}
                            (|V_\ell|-1) \varepsilon  \textnormal{ for } j \in V_\ell &\textnormal{ if } \udist(i,j) \leq d, \\
                            \varepsilon &\textnormal{ if } \udist(i,j) > d. 
                           \end{cases}
\end{equation*}
\end{proposition}
\begin{proof}
The assertion is trivial for $\udist(i,j) > d$, since then $\dist(i,j) \geq \udist(i,j) > d$ and $[f(A)^{[d]}]_{ij}= 0$ by \eqref{eq:probing_sparse_approximation}. For $i,j$ with $\udist(i,j)\leq d$ we have
 \begin{equation*}
 [f(A)^{[d]}]_{ij} =[f(A)v_\ell]_i = \sum\limits_{k\in V_\ell} [f(A)]_{ik}, \enspace \mbox{where $j\in V_\ell$. }
 \end{equation*}
 Thus,
 \begin{equation}\label{eq:entry}
  [f(A)^{[d]}]_{ij}-[f(A)]_{ij} =  \sum\limits_{\substack{k\in V_\ell \\ k\neq j}} [f(A)]_{ik}.
 \end{equation}
 If we had $\udist(i,k)\leq d$ for some $k\in V_\ell$ with $k\neq j$, then 
 \begin{equation}\label{eq:distancerelation}
  \udist(j,k) \leq \udist(j,i) + \udist(i,k) = \udist(i,j) + \udist(i,k) \leq 2d,
 \end{equation}
which is a contradiction to $j,k\in V_\ell$. Thus $\dist(i,k) \geq \udist(i,k) > d$, and therefore we have
 $$|[f(A)]_{ij}-[f(A)^{[d]}]_{ij}| \leq  \sum\limits_{\substack{k\in V_\ell \\ k\neq j}} \varepsilon = (|V_\ell|-1)\varepsilon,$$
which concludes the proof.
\end{proof}

Note that $\udist(i,j) = \udist(j,i)$ is crucial in \eqref{eq:distancerelation}, and that we do not necessarily have that $\dist(j,k) \leq \dist(i,j) + \dist(i,k)$ for the distances in the {\em directed} graph. This is why for a structurally non-symmetric matrix the proposition has to rely on a coloring of the undirected graph rather than the directed one. 

Proposition~\ref{prop:errorboundentry} immediately implies bounds for the 1-, 2- and Frobenius norms.

\begin{corollary}\label{cor:bound_matrix_norms}
Let the assumptions of Proposition~\ref{prop:errorboundentry} hold and let $\gamma=\max_{\ell} |V_\ell|$. Then with $\varepsilon = Cq^d$ we have
\begin{equation}\label{eq:sparseApprox_error_frobenius}
\| f(A)-f(A)^{[d]} \|_2 \leq \| f(A)-f(A)^{[d]} \|_F \leq n(\gamma-1) \varepsilon
\end{equation}
and
\begin{equation}\label{eq:sparseApprox_error_1norm}
\| f(A)-f(A)^{[d]} \|_1 \leq n(\gamma-1) \varepsilon.
\end{equation}
\end{corollary}

For a family of matrices $\{A_s\}_{s \in \mathcal{S}}, A_s \in \C^{n_s \times n_s}$ with uniform exponential decay \eqref{eq:generalDecaybound}, the bounds in \eqref{eq:sparseApprox_error_frobenius} and \eqref{eq:sparseApprox_error_1norm} are at least of order $\OO(n_s\varepsilon)$. If, similarly to Theorem~\ref{the:sparseApprox_general}, we assume that the level sets are polynomially bounded, the bound for the 1-norm can be made independent of $n_s$.

\begin{theorem}\label{the:error_bound_polynomially_bounded}
 Let $\{A_s\}_{s \in \mathcal{S}}$ be a family of $n_s \times n_s$ matrices such that $f(A_s)$ has uniform exponential decay \eqref{eq:generalDecaybound}.
%
Assume that the sizes of the level sets $L_s^{(\delta)}(j)$ of the undirected graphs $|G(A_s)|$ satisfy
 $$
 |L_s^{(\delta)}(j)| \leq K \, \delta^\alpha \enspace \mbox{for all nodes $j=1,\ldots,n_s$}
 $$
 and let $f(A_s)^{[d]}$ be the approximation defined by \eqref{eq:probing_sparse_approximation} with probing vectors resulting from a distance $2d$-coloring of $|G(A_s)|$. Then with $\varepsilon = Cq^d$ there exists $\widetilde{d}$ independent of $s$ such that for $d \geq \widetilde{d}$ we have
 $$\|f(A_s)-f(A_s)^{[d]}\|_1 \leq \varepsilon \mbox{ for all $s \in \mathcal{S}$ }.
 $$
\end{theorem}


\begin{proof}
For every $d>0$ we have 
\begin{eqnarray}
 \lefteqn{\hspace*{-3em} \|f(A_s)-f(A_s)^{[d]}\|_1 \, = \, \max_{j=1}^{n_s} \sum\limits_{i=1}^{n_s} |[f(A_s)]_{ij} - [f(A_s)^{[d]}]_{ij}|} \nonumber \\
  &\leq& \max_{j=1}^{n_s} \Big(\sum\limits_{\substack{i \\ \udist_s(i,j)>d}} |[f(A)]_{ij}| + \sum\limits_{\substack{i \\ \udist_s(i,j)\leq d}} \sum\limits_{\substack{k\in V_{\ell(j)} \\ k\neq j}} |[f(A)]_{ik}| \Big) \nonumber \qquad (j \in V_{\ell(j)}) \\
 &\leq&  \max_{j=1}^n \Big(\sum\limits_{\delta=d+1}^{n_s-1} |L_s^{(\delta)}(j)| C q ^ \delta +  \sum\limits_{\substack{i \\ \udist_s(i,j)\leq d}} \sum\limits_{\delta=d+1}^{n_s-1} |L_s^{(\delta)}(i)| C q^{\delta} \Big) \nonumber \\ 
  &\leq&  \max_{j=1}^n \Big(\sum\limits_{\delta=d+1}^{n_s-1} K \delta^{\alpha} C q ^ \delta +  \sum\limits_{\substack{i \\ \udist_s(i,j)\leq d}} \sum\limits_{\delta=d+1}^{n_s-1} K\delta^{\alpha} C q^{\delta} \Big) \nonumber \\ 
   &=&  \max_{j=1}^n \Big(\sum\limits_{\delta=d+1}^{n_s-1} K \delta^{\alpha} C q ^ \delta +  \sum\limits_{\rho=0}^{d} |L_s^{(\rho)}(j)| \cdot \sum\limits_{\delta=d+1}^{n_s-1} K\delta^{\alpha} C q^{\delta} \Big)
    \label{eq:errorboundind} \\
  &\leq&  \sum\limits_{\delta=d+1}^{\infty}K \delta^{\alpha} C q ^ {\delta} + \big(1+\sum\limits_{\rho=1}^d K \rho^{\alpha}\big) \sum\limits_{\delta=d+1}^{\infty} K \delta^{\alpha} C q^{\delta}, \nonumber
\end{eqnarray}
where we used \eqref{eq:entry} for the second line and that $\udist_s(i,k)>d$ for $k\in V_{\ell(j)}, k\neq j$, see \eqref{eq:distancerelation}, 
for the third.
As shown in the proof of Theorem~\ref{the:sparseApprox_general} there exists $d_1$ such that 
$$ \sum\limits_{\delta=d+1}^{\infty} K \delta^{\alpha} C q ^ \delta \leq \widehat{C} q^{\frac{d+1}{2}}   \enspace \mbox{ for $d > d_1$},
$$
with $\widehat{C}= \frac{CK}{1-\sqrt{q}}$. Hence, we obtain
$$ \|f(A_s)-f(A_s)^{[d]}\|_1  \leq \left( 2+ \sum\limits_{\rho=1}^d K \rho^\alpha \right ) \widehat{C} q^{\frac{d+1}{2}} \enspace \mbox{ for $d>d_1$.}
$$  
Since $\sum\limits_{\rho=1}^d \rho^{\alpha}  < d^{\alpha+1}$, we can find $d_2$ such that for $d > d_2$ we have
\begin{equation*}
\Big( 2+ \sum\limits_{\rho=1}^d K \rho^\alpha \Big) \widehat{C} q^{\frac{d+1}{2}}<\varepsilon.
\end{equation*}
The assertion thus holds for $\widetilde{d}=\max\{d_1,d_2\}.$
\end{proof}

While the formulation of Theorem~\ref{the:error_bound_polynomially_bounded} is focused on the uniform approximation property, we can also directly use \eqref{eq:errorboundind} to obtain error bounds for a single matrix $A$. 
We illustrate this for $\beta$-banded matrices, where---as opposed to the result formulated in Corollary~\ref{cor:bound_matrix_norms}---we now obtain a bound for the 1-norm that does not depend on $n$.

\begin{corollary}\label{cor:1norm_banded}
 Let $A\in \Cnn$ be a $\beta$-banded matrix and let $f(A)$ have exponential decay \eqref{eq:generalDecaybound}. 
Let $f(A)^{[d]}$ be the approximation defined by \eqref{eq:probing_sparse_approximation} with probing vectors resulting from the coloring \eqref{eq:coloring_banded}. Then
 $$\|f(A)-f(A)^{[d]}\|_1\leq 2 \beta q\frac{2+2d\beta}{1-q} \varepsilon, \enspace \mbox{ where $\varepsilon = Cq^d$}.$$
\end{corollary}
\begin{proof}
 For all nodes $j$ and levels $\delta$ we have $|L^{(\delta)}(j)|\leq 2\beta$. Using this and \eqref{eq:errorboundind} the approximation error of $f(A)^{[d]}$ can be bounded as
 \begin{align*}
  \|f(A)-f(A)^{[d]}\|_1&\leq \sum\limits_{\delta=d+1}^{n-1} 2 \beta C q ^ \delta + \sum\limits_{\rho=0}^d  |L^{\rho}(j)| 
  \sum\limits_{\delta=d+1}^{n-1} 2 \beta C q^\delta  \\
  &\leq 2\beta C(1+1+2d\beta) \sum\limits_{\delta=d+1}^{n-1} q^\delta \leq  2\beta C(2+2d\beta) \sum\limits_{\delta=d+1}^{\infty} q^\delta,
 \end{align*}
which concludes the proof.
\end{proof}

Another situation in which it is possible to improve upon the result of Corollary~\ref{cor:bound_matrix_norms}, now for the Frobenius norm, is when the decay bounds that we have available have their origin in a polynomial approximation property~\eqref{eq:polynomial_approximation_decay}.

\begin{theorem}\label{the:error_bound_sparse_approx_poly}
Let $A \in \Cnn$ and  assume that the function $f$ fulfills~\eqref{eq:polynomial_approximation_decay}. Let $f(A)^{[d]}$ be the approximation  defined by \eqref{eq:probing_sparse_approximation} with probing vectors $v_\ell$ resulting from a distance $2d$-coloring of $|G(A)|$ with color classes $V_\ell$. 
Then 
\begin{equation*}\label{eq:sparse_approximation_poly_theorem}
\|f(A)-f(A)^{[d]}\|_F \leq 2K\sqrt{n}\varepsilon, \enspace \mbox{ with $\varepsilon = Cq^d$},
\end{equation*}
where $K = 1$ when $A$ is normal and $K = 1+\sqrt{2}$ otherwise.
\end{theorem}
\begin{proof}
Let $p_d$ be a polynomial of degree $d$ such that $|f(z)-p_d(z)| \leq Cq^d$ for all $z \in \W(A)$, which exists since $f$ satisfies~\eqref{eq:polynomial_approximation_decay}. We now estimate the two terms in the triangle inequality
\begin{equation}\label{eq:errornorm_triangle}
\|f(A)-f(A)^{[d]}\|_F \leq \|f(A)-p_d(A)\|_F + \|f(A)^{[d]}-p_d(A)\|_F
\end{equation}
individually. For the first term, note that for $i=1,\ldots,n$ we have $\|f(A)e_i-p_d(A)e_i\|_2 \leq \|f(A)-p_d(A)\|_2 \leq K\varepsilon$ due to  Theorem~\ref{the:crouzeix}. This directly implies 
\begin{equation}\label{eq:estimate_first_term_poly}
\|f(A)-p_d(A)\|_F \leq K \sqrt{n} \varepsilon.
\end{equation}
Similarly, 
we also have $\|f(A)v_\ell - p_d(A)v_\ell\|_2  \leq K\varepsilon \|v_\ell\|_2 =  K\varepsilon\sqrt{|V_\ell|}$. For the degree $d$ polynomial $p_d$  
the sparse approximation $p_d(A)^{[d]}$ is exact so that
\begin{align*}
\|f(A)^{[d]} - p_d(A)\|_F^2 \, = \, \|f(A)^{[d]} - p_d(A)^{[d]}\|_F^2 & = \, \sum_{\ell = 1}^m \|f(A)v_\ell - p_d(A)v_\ell\|_2^2 \\
& \leq \, \sum_{\ell=1}^m K^2\varepsilon^2 |V_\ell| \, = \, K^2\varepsilon^2 n,
\end{align*}
which gives the estimate
\begin{equation}\label{eq:estimate_second_term_poly}
\|f(A)^{[d]} - p_d(A)\|_F = \|f(A)^{[d]} - p_d(A)^{[d]}\|_F \leq K \sqrt{n} \varepsilon.
\end{equation}
Inserting~\eqref{eq:estimate_first_term_poly} and \eqref{eq:estimate_second_term_poly} into~\eqref{eq:errornorm_triangle} gives the desired result.
\end{proof}

\section{Approximation of the trace of matrix functions}  \label{sec:trace}

We now turn to investigating the accuracy of the probing method \eqref{eq:probing_trace_approximation} for approximating the trace $\tr(f(A))$. As we will see, instead of using distance-$2d$ colorings for the undirected graph $|G(A)|$ we can now work with distance-$d$ colorings in the directed graph $G(A)$. 
For the probing vectors defined in~\eqref{eq:probing_vectors} we have
$$v_{\ell}^Hf(A)v_{\ell}=\sum\limits_{i\in V_\ell} [f(A)]_{ii} + \sum\limits_{\substack{i,j\in V_\ell \\ i\neq j}} [f(A)]_{ij},$$
from which we immediately obtain
\begin{equation*}\label{eq:trace_approx_with_error}
\tr(f(A)) = \sum\limits_{i=1}^n [f(A)]_{ii} = \sum\limits_{\ell=1}^m v_\ell^H f(A) v_\ell - \sum\limits_{\ell=1}^m \sum\limits_{\substack{i,j\in V_\ell \\ i\neq j}} [f(A)]_{ij}.
\end{equation*}
Thus, the error of the approximation $\T(f(A))$ from in~\eqref{eq:probing_trace_approximation} is given by
\begin{equation}\label{eq:error}
\left| \tr(f(A))-\T(f(A))\right|=\bigg|\sum\limits_{\ell=1}^m \sum\limits_{\substack{i,j\in V_\ell \\ i\neq j}} [f(A)]_{ij}\bigg|.
\end{equation}

To obtain bounds for the error~\eqref{eq:error} when $f(A)$ has exponential decay \eqref{eq:generalDecaybound}, consider a distance-$d$ coloring of $G(A)$ with color classes 
$V_\ell, \ell = 1,\dots,m$. Then, with the corresponding probing vectors~\eqref{eq:probing_vectors} and with $\varepsilon = Cq^d$, an immediate error bound is given by
\begin{equation}\label{eq:errorbound}
 \left| \tr(f(A))-\T(f(A))\right|\leq \sum\limits_{\ell=1}^m \sum\limits_{\substack{i,j\in V_\ell \\ i\neq j}} \varepsilon = \sum\limits_{\ell=1}^m |V_{\ell}|(|V_{\ell}|-1) \varepsilon.
\end{equation}
If we assume that the size of the color classes is asymptotically given by $\OO(\frac{n}{m})$, i.e., if the nodes are distributed uniformly among the color classes, and if the number of colors $m$ is independent of $n$, then the error bound \eqref{eq:errorbound} is of order $\OO(n^2) \varepsilon$. In the following we discuss cases in which better error bounds than~\eqref{eq:errorbound} can be obtained. Similar to the sparse approximation discussed in Section~\ref{subsec:sparse_approx_probing},  we can give $\OO(n)\varepsilon$ error bounds by exploiting knowledge about the specific coloring of $G(A)$. E.g., for banded matrices $A$, using the coloring \eqref{eq:coloring_banded}, we obtain the following improved error bound. Note that the result also holds for matrices $A$ for which a permutation $P^T\!\!AP$ is banded if we permute the probing vectors accordingly.

\begin{theorem}\label{the:errorbound_banded}
 Assume that $A \in \Cnn$ is $\beta$-banded and that $f(A)$ has  exponential decay \eqref{eq:generalDecaybound}. 
 Let $\T(f(A))$ be the approximation \eqref{eq:probing_trace_approximation} to the trace, where the vectors $v_\ell$ are computed with respect to the coloring \eqref{eq:coloring_banded} for a given distance $d$ and put $\varepsilon = Cq^d$. Then
 $$|\tr(f(A))-\T(f(A))| \leq \varepsilon \frac{2n}{1-q^d}.$$
\end{theorem}
\begin{proof}
The color classes of the coloring \eqref{eq:coloring_banded} are given as 
 $$V_\ell = \left \{ \ell +k(d\beta+1), k=0, \ldots, \left \lfloor \frac{n-\ell}{d\beta+1} \right \rfloor \right \}, \quad \ell=1,\ldots, d\beta+1 =:m.$$
By 
inserting the decay bounds into \eqref{eq:error}, we obtain
 \begin{align} \label{eq:banded_trace_first_estimate}
  \left| \tr(f(A))-\T(f(A))\right|\leq \sum\limits_{\ell=1}^m \sum\limits_{\substack{i,j\in V_\ell \\ i\neq j}} |[f(A)]_{ij}| 
  \leq \sum\limits_{\ell=1}^m \sum\limits_{\substack{i,j\in V_\ell \\ i\neq j}} C q^{\dist(i,j)}.
 \end{align}
 %
Now, for any color $\ell$, if $i,j \in V_\ell, i= \ell + rm, j =\ell+sm$, then
 $ 
 \dist(i,j) = |r-s|d.
 $ 
 Thus, for all $\ell$ we have
 \begin{eqnarray*}
 \sum\limits_{\substack{i,j \in V_\ell \\ i \neq j}} q^{\dist(i,j)} &=& \sum\limits_{r=0}^{\lfloor \frac{n-\ell}{m} \rfloor} \sum\limits_{\substack{s=0 \\ s\neq r}}^{\lfloor \frac{n-\ell}{m} \rfloor} q^{|s-r|d} 
 \, \leq \, 
 2 \sum\limits_{r=0}^{\lfloor \frac{n-\ell}{m} \rfloor} \sum\limits_{k=1}^{\lfloor \frac{n-\ell}{m} \rfloor} q^{kd} \\
 & \leq & 2 \sum\limits_{r=0}^{\lfloor \frac{n-\ell}{m} \rfloor} \frac{q^d}{1-q^d} 
 \, = \, 
 2 |V_\ell| \frac{q^d}{1-q^d}.
 \end{eqnarray*}
 Inserting this relation into \eqref{eq:banded_trace_first_estimate} gives 
 \begin{align*} 
  \left| \tr(f(A))-\T(f(A))\right|\leq  \sum\limits_{\ell=1}^m  |V_\ell| 2 C \frac{q^d}{1-q^d} = 2nC \frac{q^d}{1-q^d} = \varepsilon \frac{2n}{1-q^d}.
 \end{align*}
 \hfill
%
%
\end{proof}

A similar $\OO(nq^d)$ 
bound can be formulated if $G(A)$ is a regular $D$-dimensional lattice and the coloring of Theorem~\ref{the:lattice_col} is used. We state this results using the polylogarithm %
%
$
\Li_{s}(z) = \sum_{i=1}^{\infty} \frac{z^i}{i^s}.
$
 
\begin{theorem}\label{the:errorbound_lattice}
 Let $A\in \Cnn$ be a matrix for which $G(A)$ is a regular $D$-dimensional lattice. Let $f(A)$ have exponential decay \eqref{eq:generalDecaybound}. 
 Let $\T(f(A))$ be defined by \eqref{eq:probing_trace_approximation}, where the vectors $v_\ell$ are computed with respect to the distance-$d$ coloring of Theorem~\ref{the:lattice_col}. Then
 $$|\tr(f(A))-\T(f(A))| \leq 2CDn\Li_{1-D}(q^d).$$
\end{theorem}
\begin{proof}
Again, 
 \begin{align*}
  \left| \tr(f(A))-\T(f(A))\right|\leq \sum\limits_{\ell=1}^k \sum\limits_{\substack{i,j\in V_\ell \\ i\neq j}} |[f(A)]_{ij}| 
  \leq \sum\limits_{\ell=1}^k \sum\limits_{\substack{i,j\in V_\ell \\ i\neq j}} C q^{\dist(i,j)},
 \end{align*}
with the color classes $V_\ell$ from Theorem~\ref{the:lattice_col}.
For this coloring, as illustrated in Figure~\ref{fig:graph_lattice_col}, the distances between nodes from the same color class are multiples of $d$ and these nodes actually form again a regular $D$-dimensional lattice. Lemma~\ref{lemma:levelsets} shows that for each node the number of nodes with distance $\delta$ in this lattice, i.e., with distance $\delta d$ in the original lattice, is bounded by $2D \, \delta^{D-1}$. Thus
\begin{eqnarray*}
\sum\limits_{\ell=1}^k \sum\limits_{\substack{i,j\in V_\ell \\ i\neq j}} C q^{\dist(i,j)} &\leq& \sum\limits_{\ell=1}^k |V_\ell| \sum\limits_{\delta=1}^\infty 2D \, \delta^{D-1} \, Cq^{\delta d} \\
&\leq& 2CDn \sum\limits_{\delta=1}^\infty \delta^{D-1}q^{\delta d}\\
&=& 2CDn\Li_{1-D}(q^d).\label{eq:sum_polylogarithm}
\end{eqnarray*}
\hfill
\end{proof}

\begin{remark}\label{rem:polylogarithm_rational}
For a given value of $D$, the bound from Theorem~\ref{the:errorbound_lattice} can be cast into a more explicit form by noting that all polylogarithms of negative integer order are rational functions of the form $\Li_{-s}(z) = \frac{p_{s}(z)}{(1-z)^{s+1}}$ where $p_s$ is a polynomial of degree $s$ such that $p_s(0) = 0$. An explicit representation can be found in terms of Eulerian numbers; see, e.g.,~\cite{Miller2008}. In particular, the first few polylogarithms of negative integer order are given by
$$\Li_{-1}(z) = \frac{z}{(1-z)^2}, \qquad \Li_{-2}(z) = \frac{z+z^2}{(1-z)^3}, \qquad \Li_{-3}(z) = \frac{z+4z^2+z^3}{(1-z)^4}.$$
Using these relations we, e.g., find the following bound for $D = 4$
$$|\tr(f(A))-\T(f(A))| \leq 8Cn\frac{q^d+4q^{2d}+q^{3d}}{(1-q^d)^4},$$
which for large $d$ behaves like $8Cnq^d$.

Let us further note that for the case $D=1$, i.e., a tridiagonal matrix $A$, both Theorem~\ref{the:errorbound_banded} and Theorem~\ref{the:errorbound_lattice} are applicable. Since $ \Li_0(z) = \tfrac{z}{1-z},$ both theorems actually agree in this case.
\end{remark}

As in the situation where we looked at the approximation quality for the matrix function as a whole, 
we can again derive improved error bounds when we have a polynomial approximation property~\eqref{eq:polynomial_approximation_decay} available.

\begin{theorem}\label{the:boundsum}
Let $A \in \Cnn$ and assume that $f$ fulfills  condition~\eqref{eq:polynomial_approximation_decay}. Let the approximation $\T(f(A))$ in \eqref{eq:probing_trace_approximation} be obtained using a distance-$d$ coloring of $G(A)$. 
Then, with $\varepsilon = Cq^d$ we have
 \begin{equation}\label{eq:errorbound2}
 \left| \tr(f(A))-\T(f(A))\right|\leq 2Kn \varepsilon,
\end{equation}
where $K = 1$ when $A$ is normal and $K = 1+\sqrt{2}$ otherwise.
\end{theorem}
\begin{proof} We proceed as in the proof of Theorem~\ref{the:error_bound_sparse_approx_poly}.
Let $p_d$ be a polynomial of degree $d$ such that $|f(z)-p_d(z)| \leq Cq^d = \varepsilon$ for all $z \in \W(A)$, which exists since $f$ fulfills~\eqref{eq:polynomial_approximation_decay}. Then
\begin{equation} \label{eq:poly_approx_matrix_error}
\|f(A)-p_d(A)\|_2 \leq K\varepsilon.
\end{equation}
We write
\begin{equation}\label{eq:trace_error_triangle}
\left| \tr(f(A))-\T(f(A))\right| \leq \left| \tr(f(A))-\tr(p_d(A))\right| + \left| \T(f(A))-\tr(p_d(A))\right|.
\end{equation}
For the first term, we get, using the linearity of the trace, \eqref{eq:poly_approx_matrix_error} and the Cauchy-Schwarz inequality
\begin{equation}\label{eq:trace_first_term}
\left| \tr(f(A))-\tr(p_d(A))\right| \leq \sum\limits_{\ell=1}^n \left|e_\ell^T\left(f(A)-p_d(A)\right)e_\ell\right| \leq \sum\limits_{\ell=1}^n K\varepsilon = Kn\varepsilon.
\end{equation}
For the second term, note that the probing approximation $\T(p_d(A))$ for the trace is exact, $\tr(p_d(A))=\T(p_d(A))$. Therefore, in a similar manner as for the first term, we obtain
\begin{eqnarray}
\left| \T(f(A))-\tr(p_d(A))\right| &=& \left| \T(f(A))-\T(p_d(A))\right| \leq \sum\limits_{\ell=1}^m \left|v_\ell^T\left(f(A)-p_d(A)\right)v_\ell\right| \nonumber\\
&\leq& \sum\limits_{\ell=1}^m K|V_\ell|\varepsilon = Kn\varepsilon. \label{eq:trace_second_term}
\end{eqnarray}
Inserting~\eqref{eq:trace_first_term} and~\eqref{eq:trace_second_term} into~\eqref{eq:trace_error_triangle} concludes the proof.
\end{proof}

The numerical examples in Section~\ref{sec:experiments} illustrate that the error of the probing-based approximations scales indeed linearly with the dimension $n$ of the matrix. In this sense, $\OO(n) \varepsilon$ error bounds are the best we can achieve.

\section{Using Krylov subspace methods in the probing approach}\label{sec:krylov}

Probing methods require the computation of matrix-vector products $f(A)v_\ell$ or bilinear forms $v_\ell^Tf(A)v_\ell$. Both are standard tasks in numerical linear algebra, for which a plethora of different methods has been developed.  Widely used methods for both tasks are Krylov subspace methods. As with any iterative method, an important question arising in this context is how to find a good stopping criterion in order to
keep the computational cost as small as possible while at the same time guaranteeing that the desired overall accuracy is reached in the approximation of $f(A)$ or $\tr(f(A))$. 

We now answer this question for the situation that the decay bounds we have available stem from a polynomial approximation property of the form~\eqref{eq:polynomial_approximation_decay}. 
We begin by very shortly reviewing a few important facts about Arnoldi's 
method, the prototype Krylov subspace method; see, e.g.,~\cite[Section 3.5]{FrommerSimoncini2008} or~\cite[Section 13.2]{Higham2008} for details. The approximation for $f(A)b$ from $s$ steps of Arnoldi's method is given by
\begin{equation}\label{eq:arnoldi_approximation}
f_s = \|b\|_2 \review{Q_s} f(H_s) e_1,
\end{equation}
where the columns of $\review{Q_s}$ are the orthonormal Arnoldi basis vectors and $H_s =\review{Q_s^H}A\review{Q_s} \in \C^{s \times s}$ is the upper Hessenberg matrix containing the orthogonalization coefficients. We have that 
\begin{equation*}\label{eq:arnoldi_polynomial}
f_s = \|b\|_2 \review{Q_s} 
\widetilde{p}_{s-1}(H_s)e_1 = \widetilde{p}_{s-1}(A)b,
\end{equation*}
where $\widetilde{p}_{s-1}$ is the polynomial of degree $s-1$ that interpolates $f$ on the eigenvalues of $H_s$ in the Hermite sense. 

\review{The near-optimality property of the Arnoldi approximation~\cite[Proposition~3.1]{BR09}, \cite[Section 4.2.2]{Guettel2010}, which is based on Theorem~\ref{the:crouzeix}, guarantees that it fulfills
$$\|f(A)b-f_s\|_2 \leq 2K\|b\| \min_{p_s \in \Pi_s} \max_{z \in \mathcal{W}} |f(z)-p(z)|$$
for any compact set $\mathcal{W}$ containing the numerical range of $A$. Thus, if the assumption~\eqref{eq:polynomial_approximation_decay} holds, this implies 
\begin{equation}\label{eq:arnoldi_error}
\|f(A)b-f_s\|_2 \leq 2\|b\|_2KCq^{s-1}.
\end{equation}}

\subsection{
Sparse approximation}\label{subsec:krylov_sparseapprox}
Let $f_s^{(\ell)}$ denote the Arnoldi approximation~\eqref{eq:arnoldi_approximation} for $f(A)v_\ell$. By replacing $f(A)v_\ell$ by $f_s^{(\ell)}$ in~\eqref{eq:probing_sparse_approximation}, we obtain the approximation
\begin{equation}\label{eq:sparseApprox_krylov_new}
 [\widetilde{f(A)}^{[d]}]_{ij}:= \begin{cases}
                            [f_{s}^{(\ell)}]_i \text{ for } j \in V_\ell &\text{ if } \udist(i,j) \leq d, \\
                            0 &\text{ if } \udist(i,j) > d,
                           \end{cases}
\end{equation}
where the color classes $V_\ell$ come from a distance-$2d$ coloring of $|G(A)|$.
In the triangle inequality
\begin{equation}\label{eq:combined_error_estimate_standard}
	\|f(A) - \widetilde{f(A)}^{[d]} \| \leq \|f(A) - f(A)^{[d]}\| + \|f(A)^{[d]} - \widetilde{f(A)}^{[d]}\|, 
\end{equation}
Theorem~\ref{the:error_bound_sparse_approx_poly} shows that for the Frobenius norm the first term in~\eqref{eq:combined_error_estimate_standard} can be bounded as
\begin{equation} \label{eq:first_term_Arnoldi}
\|f(A) - \widetilde{f(A)}^{[d]} \|_F \leq 2KCq^d\sqrt{n}.
\end{equation}
For the second term, note that
\begin{equation*}\label{eq:frobenius_twonorm}
\|f(A)^{[d]} - \widetilde{f(A)}^{[d]}\|_F^2 = \sum_{\ell = 1}^m \|f(A)v_\ell - f_s^{(\ell)}\|_2^2,
\end{equation*}
so that inserting~\eqref{eq:arnoldi_error}, we obtain
\begin{eqnarray}
\|f(A)^{[d]} - \widetilde{f(A)}^{[d]}\|_F &\leq& \Big(\sum_{\ell = 1}^m 4\|v_\ell\|_2^2K^2C^2q^{2(s-1)}
\Big)^{1/2} \nonumber \\
&=& \Big(4K^2C^2q^{2(s-1)}
\sum_{\ell = 1}^m |V_\ell| \Big)^{1/2} \nonumber \\
&=& 2KCq^{(s-1)}\sqrt{n} \label{eq:second_term_Arnoldi}.
\end{eqnarray}
Inequalities \eqref{eq:first_term_Arnoldi} and \eqref{eq:second_term_Arnoldi}, on the one hand, give us the final estimate
\begin{equation}\label{eq:combined_error_estimate_frobenius}
	\|f(A) - \widetilde{f(A)}^{[d]} \|_F \leq 2KC\sqrt{n}(q^d + q^{s-1}). 
\end{equation}
On the other hand, they also \review{show that} after $d+1$ Arnoldi steps we can expect the Krylov approximation error to have the same magnitude as the probing error. If we perform more than $s=d+1$ Arnoldi steps, the overall error is likely to be dominated by the probing error, so that further Arnoldi iterations will have no or little effect on the overall error.
Choosing $s = d+1$ the overall bound simplifies to
\begin{equation*}
    \|f(A) - \widetilde{f(A)}^{[d]} \|_F \leq 4K\sqrt{n}\varepsilon.
\end{equation*}

As we will illustrate in the numerical experiments in 
Section~\ref{sec:experiments}, performing more than $d+1$ Arnoldi steps 
does typically indeed not lead to any further reduction of the overall error.  Heuristically this can be further motivated as follows: The entries of the vector $f(A)v_\ell$ that we approximate by the Arnoldi iterates \review{do} not contain the exact entries of $f(A)$, but perturbed entries due to the ``mixing'' of contributions from nodes of the same color. Until the $d+1$st iteration of the Arnoldi method, this mixing does not occur in the basis vectors; see Figure~\ref{fig:spread} for an illustration. For $s > d+1$, then, the additional accuracy with which we approximate $f(A)v_\ell$ is spoiled by  the loss of accuracy in the approximation of $f(A)$ due to increased mixing. 

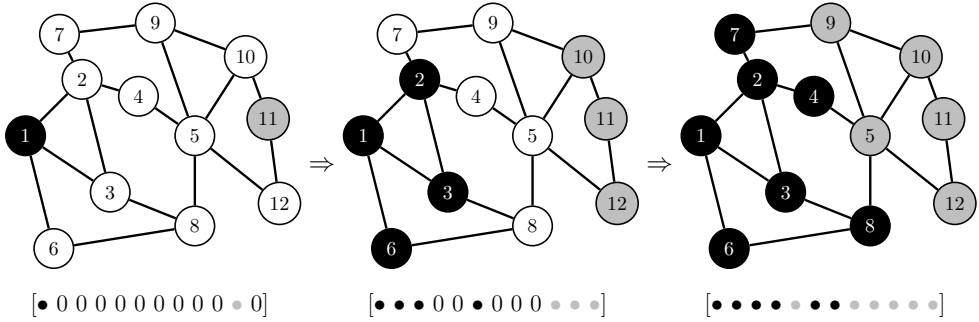
\begin{figure}
\scalebox{.75}{%
\begin{tikzpicture}[baseline={([yshift=-1ex]current bounding box.center)}]
\node[thick,circle,draw=black,fill=black, minimum width=.7cm] (1) at (0,0) {\color{white}1};
\node[thick,circle,draw=black,fill=white, minimum width=.7cm] (2) at (1,1) {\color{black}2};
\node[thick,circle,draw=black,fill=white, minimum width=.7cm] (3) at (1.5,-1) {\color{black}3};
\node[thick,circle,draw=black,fill=white, minimum width=.7cm] (4) at (2,.7) {\color{black}4};
\node[thick,circle,draw=black,fill=white, minimum width=.7cm] (5) at (3,0) {\color{black}5};
\node[thick,circle,draw=black,fill=white, minimum width=.7cm] (6) at (.5,-2) {\color{black}6};
\node[thick,circle,draw=black,fill=white, minimum width=.7cm] (7) at (.6,1.8) {\color{black}7};
\node[thick,circle,draw=black,fill=white, minimum width=.7cm] (8) at (3,-1.6) {\color{black}8};
\node[thick,circle,draw=black,fill=white, minimum width=.7cm] (9) at (2.3,2) {\color{black}9};
\node[thick,circle,draw=black,fill=white, minimum width=.7cm] (10) at (3.9,1.4) {\color{black}10};
\node[thick,circle,draw=black,fill=gray!50!white, minimum width=.7cm] (11) at (4.3,.3) {\color{black}11};
\node[thick,circle,draw=black,fill=white, minimum width=.7cm] (12) at (4.5,-1.2) {\color{black}12};

\node (vec) at (2.2,-3) {\color{black}\large [{\color{black}$\bullet$} 0 0 0 0 0 0 0 0 0 {\color{gray!50!white}$\bullet$} 0]};

\draw[black, very thick] (1) -- (2);
\draw[black, very thick] (1) -- (3);
\draw[black, very thick] (1) -- (6);
\draw[black, very thick] (2) -- (4);
\draw[black, very thick] (2) -- (7);
\draw[black, very thick] (2) -- (3);
\draw[black, very thick] (3) -- (8);
\draw[black, very thick] (4) -- (5);
\draw[black, very thick] (5) -- (8);
\draw[black, very thick] (5) -- (9);
\draw[black, very thick] (6) -- (8);
\draw[black, very thick] (7) -- (9);
\draw[black, very thick] (10) -- (5);
\draw[black, very thick] (10) -- (9);
\draw[black, very thick] (11) -- (10);
\draw[black, very thick] (11) -- (12);
\draw[black, very thick] (12) -- (5);

\end{tikzpicture}}
$\Rightarrow$
\scalebox{.75}{%
\begin{tikzpicture}[baseline={([yshift=-1ex]current bounding box.center)}]
\node[thick,circle,draw=black,fill=black, minimum width=.7cm] (1) at (0,0) {\color{white}1};
\node[thick,circle,draw=black,fill=black, minimum width=.7cm] (2) at (1,1) {\color{white}2};
\node[thick,circle,draw=black,fill=black, minimum width=.7cm] (3) at (1.5,-1) {\color{white}3};
\node[thick,circle,draw=black,fill=white, minimum width=.7cm] (4) at (2,.7) {\color{black}4};
\node[thick,circle,draw=black,fill=white, minimum width=.7cm] (5) at (3,0) {\color{black}5};
\node[thick,circle,draw=black,fill=black, minimum width=.7cm] (6) at (.5,-2) {\color{white}6};
\node[thick,circle,draw=black,fill=white, minimum width=.7cm] (7) at (.6,1.8) {\color{black}7};
\node[thick,circle,draw=black,fill=white, minimum width=.7cm] (8) at (3,-1.6) {\color{black}8};
\node[thick,circle,draw=black,fill=white, minimum width=.7cm] (9) at (2.3,2) {\color{black}9};
\node[thick,circle,draw=black,fill=gray!50!white, minimum width=.7cm] (10) at (3.9,1.4) {\color{black}10};
\node[thick,circle,draw=black,fill=gray!50!white, minimum width=.7cm] (11) at (4.3,.3) {\color{black}11};
\node[thick,circle,draw=black,fill=gray!50!white, minimum width=.7cm] (12) at (4.5,-1.2) {\color{black}12};

\node (vec) at (2.2,-3) {\color{black}\large [{\color{black}$\bullet$} {\color{black}$\bullet$} {\color{black}$\bullet$} 0 0 {\color{black}$\bullet$} 0 0 0 {\color{gray!50!white}$\bullet$} {\color{gray!50!white}$\bullet$} {\color{gray!50!white}$\bullet$}]};

\draw[black, very thick] (1) -- (2);
\draw[black, very thick] (1) -- (3);
\draw[black, very thick] (1) -- (6);
\draw[black, very thick] (2) -- (4);
\draw[black, very thick] (2) -- (7);
\draw[black, very thick] (2) -- (3);
\draw[black, very thick] (3) -- (8);
\draw[black, very thick] (4) -- (5);
\draw[black, very thick] (5) -- (8);
\draw[black, very thick] (5) -- (9);
\draw[black, very thick] (6) -- (8);
\draw[black, very thick] (7) -- (9);
\draw[black, very thick] (10) -- (5);
\draw[black, very thick] (10) -- (9);
\draw[black, very thick] (11) -- (10);
\draw[black, very thick] (11) -- (12);
\draw[black, very thick] (12) -- (5);
\end{tikzpicture}}
$\Rightarrow$
\scalebox{.75}{%
\begin{tikzpicture}[baseline={([yshift=-1ex]current bounding box.center)}]
\node[thick,circle,draw=black,fill=black, minimum width=.7cm] (1) at (0,0) {\color{white}1};
\node[thick,circle,draw=black,fill=black, minimum width=.7cm] (2) at (1,1) {\color{white}2};
\node[thick,circle,draw=black,fill=black, minimum width=.7cm] (3) at (1.5,-1) {\color{white}3};
\node[thick,circle,draw=black,fill=black, minimum width=.7cm] (4) at (2,.7) {\color{white}4};
\node[thick,circle,draw=black,fill=gray!50!white, minimum width=.7cm] (5) at (3,0) {\color{black}5};
\node[thick,circle,draw=black,fill=black, minimum width=.7cm] (6) at (.5,-2) {\color{white}6};
\node[thick,circle,draw=black,fill=black, minimum width=.7cm] (7) at (.6,1.8) {\color{white}7};
\node[thick,circle,draw=black,fill=black, minimum width=.7cm] (8) at (3,-1.6) {\color{white}8};
\node[thick,circle,draw=black,fill=gray!50!white, minimum width=.7cm] (9) at (2.3,2) {\color{black}9};
\node[thick,circle,draw=black,fill=gray!50!white, minimum width=.7cm] (10) at (3.9,1.4) {\color{black}10};
\node[thick,circle,draw=black,fill=gray!50!white, minimum width=.7cm] (11) at (4.3,.3) {\color{black}11};
\node[thick,circle,draw=black,fill=gray!50!white, minimum width=.7cm] (12) at (4.5,-1.2) {\color{black}12};

\node (vec) at (2.2,-3) {\color{black}\large [{\color{black}$\bullet$} {\color{black}$\bullet$} {\color{black}$\bullet$} {\color{black}$\bullet$} {\color{gray!50!white}$\bullet$} {\color{black}$\bullet$} {\color{black}$\bullet$} {\color{gray!50!white}$\bullet$} {\color{gray!50!white}$\bullet$} {\color{gray!50!white}$\bullet$} {\color{gray!50!white}$\bullet$} {\color{gray!50!white}$\bullet$}]};

\draw[black, very thick] (1) -- (2);
\draw[black, very thick] (1) -- (3);
\draw[black, very thick] (1) -- (6);
\draw[black, very thick] (2) -- (4);
\draw[black, very thick] (2) -- (7);
\draw[black, very thick] (2) -- (3);
\draw[black, very thick] (3) -- (8);
\draw[black, very thick] (4) -- (5);
\draw[black, very thick] (5) -- (8);
\draw[black, very thick] (5) -- (9);
\draw[black, very thick] (6) -- (8);
\draw[black, very thick] (7) -- (9);
\draw[black, very thick] (10) -- (5);
\draw[black, very thick] (10) -- (9);
\draw[black, very thick] (11) -- (10);
\draw[black, very thick] (11) -- (12);
\draw[black, very thick] (12) -- (5);

\end{tikzpicture}}
\caption{Spreading of the nonzero entries in the first three Arnoldi basis vectors, starting with $v_\ell =e_1+e_{11}$. 
\review{Entries to which only the iteration corresponding to node $1$ contributed (and the corresponding nodes of the graph) are marked in black, while entries/nodes to which only the iteration corresponding to node $11$ contributed are marked in gray.}  Because the nodes have a distance of $5$, no mixing occurs in the first 3 basis vectors.}
\label{fig:spread}
\end{figure}

\subsection{
Approximating the trace
}

The $s$-step Arnoldi approximation for a bilinear form $v_\ell^Hf(A)v_\ell$ is given by
\begin{equation}\label{eq:arnoldi_bilinear}
v_\ell^H f(A)v_\ell \approx \alpha_s^{(\ell)} := \|v_\ell\|_2^2 e_1^Hf(H_s)e_1.
\end{equation}
Using the relation between Krylov subspace methods, Gaussian quadrature and moment matching, it has been shown in~\cite{GolubMeurant1994,GolubMeurant2010,Strakos2009}, e.g., that \eqref{eq:arnoldi_bilinear} is exact if $f$ is a polynomial up to degree $2s-1$ when $A$ is Hermitian and up to degree $s$ when $A$ is non-Hermitian.  This leads to the following \review{ theorem which gives a general exposition of results from \cite[Section~3]{BR09} and~\cite[Theorem~4.2]{SaadUbaruJie2017}.}
\begin{theorem}
\review{Let $f$ and $A$ fulfill the assumptions of Theorem~\ref{the:polynomial_approx_decay}. Then the} error of the approximation~\eqref{eq:arnoldi_bilinear} satisfies
\begin{equation*}
|v_\ell^H f(A)v_\ell - \alpha_s^{(\ell)}| \, \leq \, 2 \|v_\ell\|_2^2 KC \cdot 
\begin{cases} 
q^{2s-1} & \text{when } A \text{ is Hermitian}, \\
q^{s} & \text{otherwise},
\end{cases}
\end{equation*}
where $C, q$ are as in Theorem~\ref{the:polynomial_approx_decay} and $K$ is the constant from Theorem~\ref{the:crouzeix}.
\end{theorem}
\begin{proof}
We consider just the Hermitian case; the non-Hermitian case follows analogously. Let $p_{2s-1}^\ast(z) \in \Pi_{2s-1}$ be such that
\begin{equation}\label{eq:assumption}
\max_{z \in \W(A)} |f(z)-p_{2s-1}^\ast(z)| \leq Cq^{2s-1}.
\end{equation}
As the approximation~\eqref{eq:arnoldi_bilinear} is exact for $v_\ell^H p_{2s-1}^\ast(A) v_\ell$, we have
$$
|v_\ell^Hf(A)v_\ell - \alpha_s^{(\ell)} | 
\, = \, |v_\ell^H(f(A)-p_{2s-1}^\ast(A))v_\ell - \|v_\ell\|_2^2 e_1^H(f(H_s)-p_{2s-1}^\ast(H_s))e_1|.
$$
From this, using the triangular inequality and the Cauchy-Schwarz inequality, we get
\begin{eqnarray*}
\lefteqn{|v_\ell^Hf(A)v_\ell - \|v_\ell\|_2^2 e_1^Hf(H_s)e_1|} \qquad & & \\ 
&\leq& \|v_\ell\|_2 \,  \|f(A)v_\ell - p_{2s-1}^\ast(A)v_\ell\|_2 + \|v_\ell\|_2^2 \, \|f(H_s)e_1-p_{2s-1}^\ast(H_s)e_1\|_2 \\
&\leq& \|v_\ell\|_2^2 \, \|f(A)-p_{2s-1}^\ast(A)\|_2 + \|v_\ell\|_2^2 \, \|f(H_s)-p_{2s-1}^\ast(H_s)\|_2 
\end{eqnarray*}
Now, $\|f(A)-p_{2s-1}^\ast(A)\|_2 \leq Cq^{2s-1}$ due to \eqref{eq:assumption} and \eqref{eq:crouzeix}, and the same bound applies to $\|f(H_s)-p_{2s-1}^\ast(H_s)\|_2$, since  $\W(H_s) \subseteq \W(A)$ due to $H_s = \review{Q_s^H}A\review{Q_s}$ with $\review{Q_s}$ having orthonormal columns.
\end{proof}

Thus,  choosing $s = \left\lceil\frac{d+1}{2}\right\rceil$ or $s = d$, we obtain the bound
\begin{equation*}\label{eq:krylov_trace_bound}
|\T(f(A)) - \sum\limits_{\ell=1}^m \alpha_s^{\ell}| \leq KC q^d \sum\limits_{\ell=1}^m \|v_\ell\|_2^2 = KC q^d\sum\limits_{\ell=1}^m |V_l| = 2KCnq^d.
\end{equation*}

We are therefore in the order of magnitude of the bound for the probing error given in Theorem~\ref{the:boundsum} for probing vectors coming from a distance-$d$ coloring of $G(A)$ after $d$ (or $\approx d/2$ if $A$ is Hermitian) Arnoldi steps.

\section{Numerical experiments}\label{sec:experiments}
In this section, we perform various numerical experiments both on model problems and on matrices coming from applications to investigate the quality of our error bounds, with particular emphasis on their scaling behavior with respect to growing matrix dimension $n$ and increasing probing distance $d$. \review{All experiments were implemented in \texttt{MATLAB R2020a}.} Unless explicitly stated otherwise, we compute the exact quantities $f(A)v_\ell$ and $v_\ell^Tf(A)v_\ell$ used to obtain the exact error of our approximations to machine precision, \review{using the \texttt{MATLAB} built-in functions \texttt{inv}, \texttt{sqrtm} and \texttt{logm}}.
 
\begin{figure}
\begin{subfigure}{.49\textwidth}
	\includegraphics[width=\textwidth]{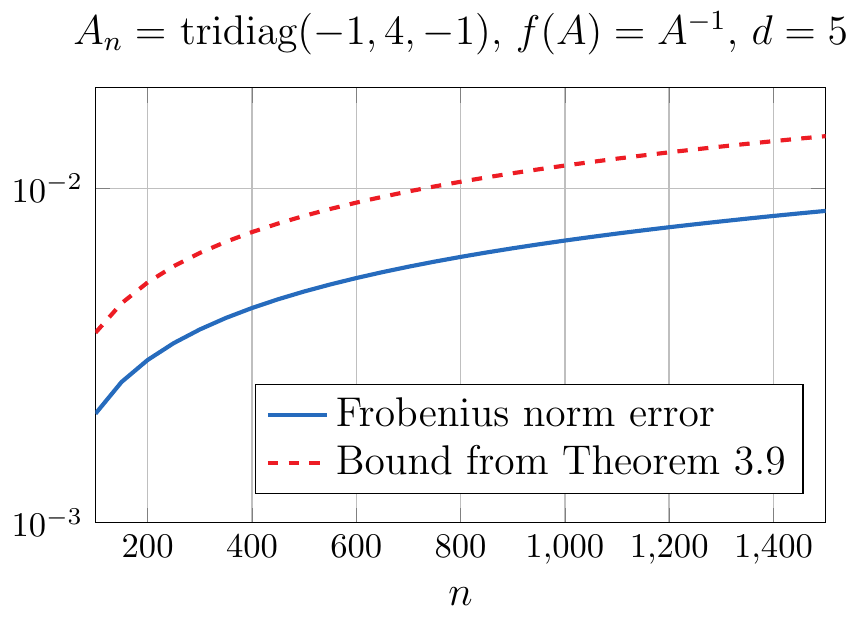}
\end{subfigure}
\begin{subfigure}{.49\textwidth}
	\includegraphics[width=\textwidth]{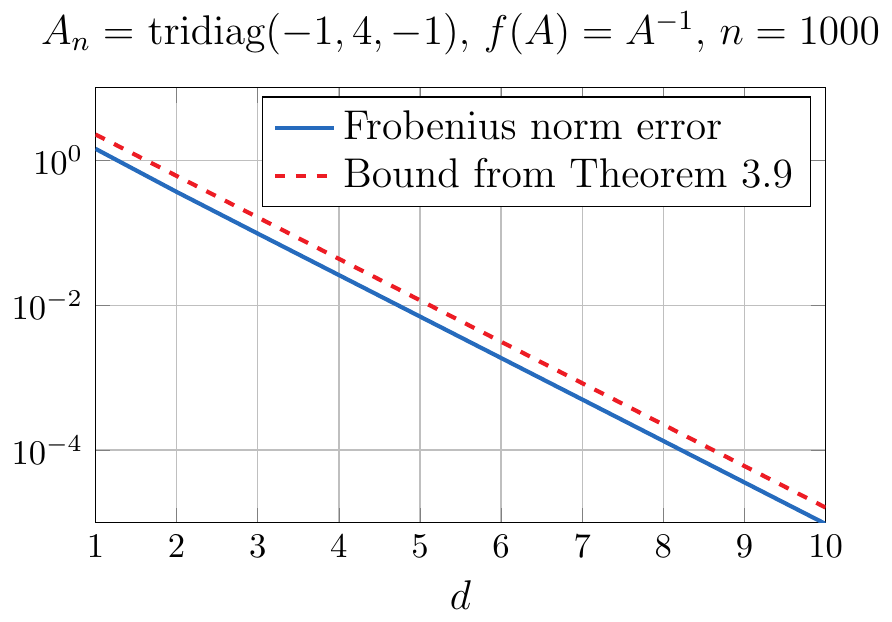}
\end{subfigure}
\caption{Actual Frobenius norm error and error bound for the sparse approximation~\eqref{eq:probing_sparse_approximation} corresponding to the coloring~\eqref{eq:coloring_banded} for the matrix $A_n = \tridiag(-1,4,-1) \in \Cnn$ and $f(z) = 1/z$ for varying $n$ (left) and $d$ (right).}\label{fig:tridiag_inv_sparse}
\end{figure}

\subsection{Tridiagonal model problem}
As a first, simple test example, following \cite{BenziSimoncini2015}, we consider the family of tridiagonal matrices $A_n = \tridiag(-1,4-1) \in \Cnn$. 
The spectra of these matrices satisfy $\spec(A_n) \subset [2,6]$ independently of $n$. 
We consider the two functions $f(z) = 1/z$ and $f(z) = z^{-1/2}$ in the following and we always use the banded matrix coloring~\eqref{eq:coloring_banded} with $\beta = 1$. 
In a first experiment, we compute sparse approximations of $A_n^{-1}$ for varying dimension $n$ while $d = 5$ is fixed and for varying $d$ while $n = 1000$ is fixed. From~\cite[Theorem 2.4]{DemkoMossSmith1984}, the entries of $A_n^{-1}$ exhibit an exponential decay with $C=\frac{1}{2}$ and $q = \tfrac{\sqrt{3}-1}{\sqrt{3+1}}$.
The actual error norms together with our error bounds from Theorem~\ref{the:error_bound_sparse_approx_poly} are depicted in Figure~\ref{fig:tridiag_inv_sparse}. In both cases, the bounds are quite 
tight and closely follow the actual error curve. We repeat the experiment for the inverse square root $A_n^{-1/2}$. The entries of this matrix function again decay exponentially, with $C=\sqrt{2}$ and $q = \tfrac{\sqrt{3}-1}{\sqrt{3+1}}$, see~\cite[Theorem 4]{FrommerSchimmelSchweitzer2018_1}. This time, we compare the actual error to the $1$-norm error bound of Corollary~\ref{cor:1norm_banded}, because the decay bound from~\cite[Theorem 4]{FrommerSchimmelSchweitzer2018_1} is not based on a polynomial approximation property of the form~\eqref{eq:polynomial_approximation_decay}. The results of this experiment are shown in Figure~\ref{fig:tridiag_invsqrt_sparse}. Again we see a good agreement between the actual error and the error bound, although it is not quite as sharp as before, overestimating the error by between one and two orders of magnitude. Still, the qualitative behavior is captured quite accurately. In particular, the $1$-norm error is independent of $n$, as predicted by our theoretical results.

\begin{figure}
\begin{subfigure}{.49\textwidth}
		\includegraphics[width=\textwidth]{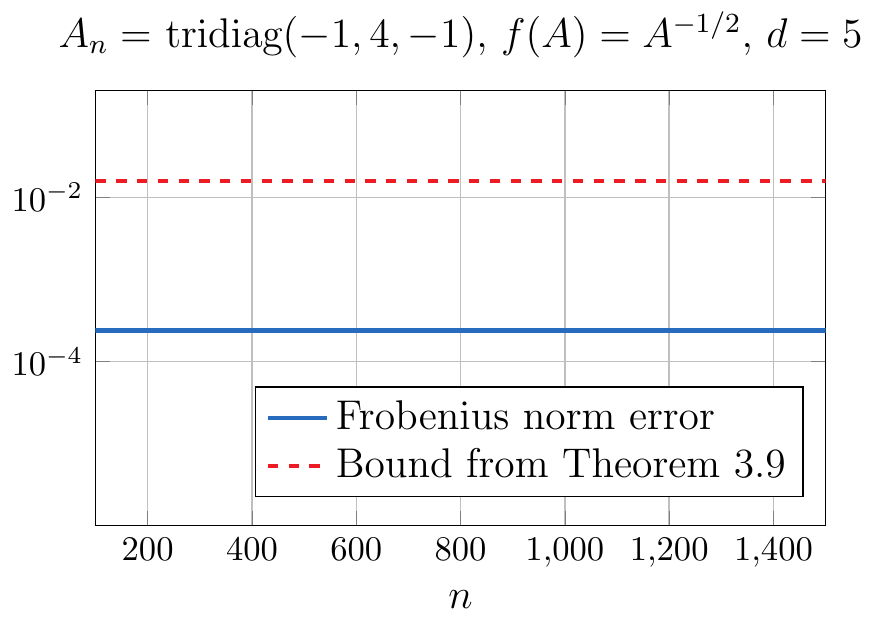}
\end{subfigure}
\begin{subfigure}{.49\textwidth}
	\includegraphics[width=\textwidth]{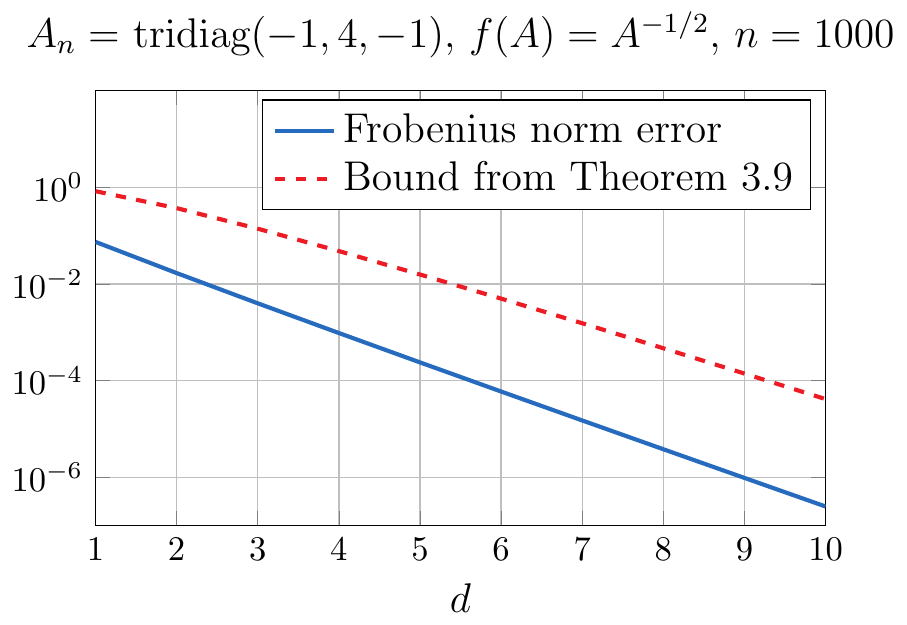}
\end{subfigure}
\caption{Actual  $1$-norm error and error bound for the sparse approximation~\eqref{eq:probing_sparse_approximation} corresponding to the coloring~\eqref{eq:coloring_banded} for the matrix $A_n = \tridiag(-1,4,-1) \in \Cnn$ and $f(z) = z^{-1/2}$ for varying $n$ (left) and $d$ (right).}\label{fig:tridiag_invsqrt_sparse}
\end{figure}

\begin{figure}
\begin{subfigure}{.49\textwidth}
	\includegraphics[width=\textwidth]{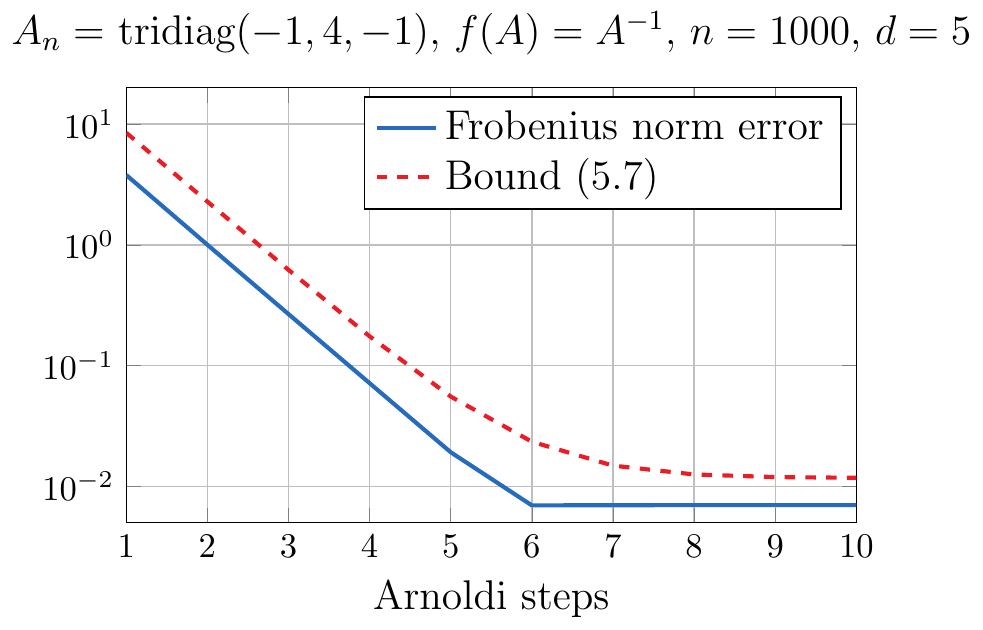}
\end{subfigure}
\begin{subfigure}{.49\textwidth}
	\includegraphics[width=\textwidth]{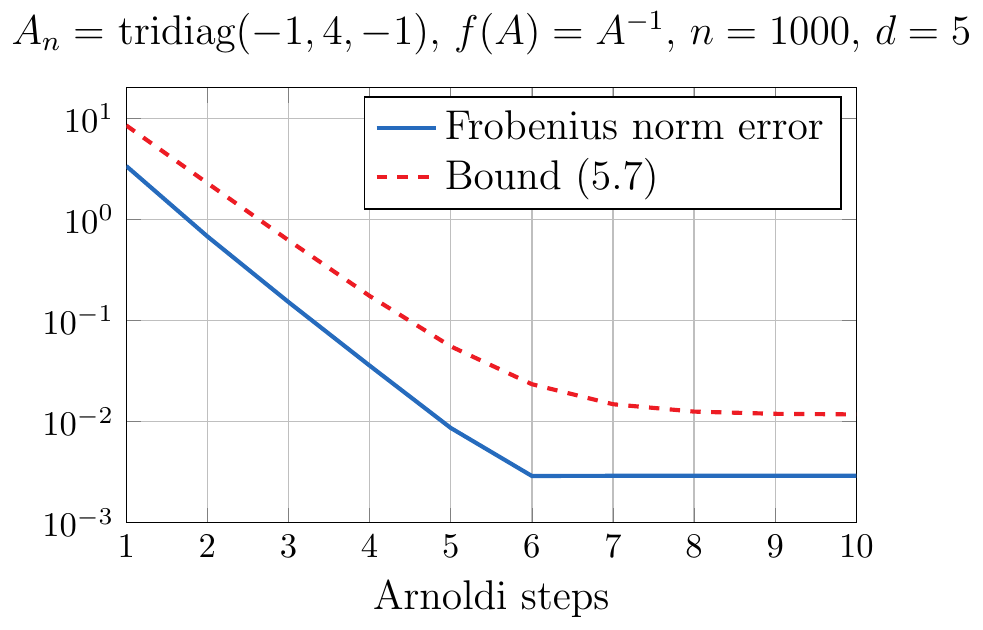}
\end{subfigure}
		
\caption{Actual Frobenius norm error and error bound (in dependence of the number of Arnoldi steps) for the sparse approximation~\eqref{eq:sparseApprox_krylov_new} corresponding to the coloring~\eqref{eq:coloring_banded} for the matrix $A_n = \tridiag(-1,4,-1) \in \Cnn$ and $f(z) = 1/z$ \review{(left) and $f(z) = z^{-1/2}$ (right)}  with $n = 1000$, and $d = 5$.}\label{fig:tridiag_invsqrt_sparse_arnoldi}
\end{figure}

We also use this example to illustrate the influence of the number of Arnoldi steps used for approximating $f(A)v_\ell$ in the approximation~\eqref{eq:sparseApprox_krylov_new}, see Figure~\ref{fig:tridiag_invsqrt_sparse_arnoldi}. We fix $n = 1000$ and $d = 5$ and compute the approximation error resulting when $s$ Arnoldi steps per vector are performed, for $s = 1,\dots, 2d$ and compare it to the bound~\eqref{eq:combined_error_estimate_frobenius}. \review{Note that for $f(z) = z^{-1/2}$ the quantity ~\eqref{eq:combined_error_estimate_frobenius} must be considered an estimate for the error rather than a bound, as the decay estimates used to obtain it are not based on a polynomial approximation property.}
We observe that \review{the bounds/estimates are} in very good agreement with the actual error, and further, that the approximation error stagnates after $s = d+1$, 
confirming our intuition explained in Section~\ref{subsec:krylov_sparseapprox} that from this point on, the increased accuracy of the Krylov approximation is counteracted by the increased mixing between contributions of nodes from the same color class, so that no further decrease of the overall approximation error can be expected.

\begin{figure}
\begin{subfigure}{.49\textwidth}
	\includegraphics[width=\textwidth]{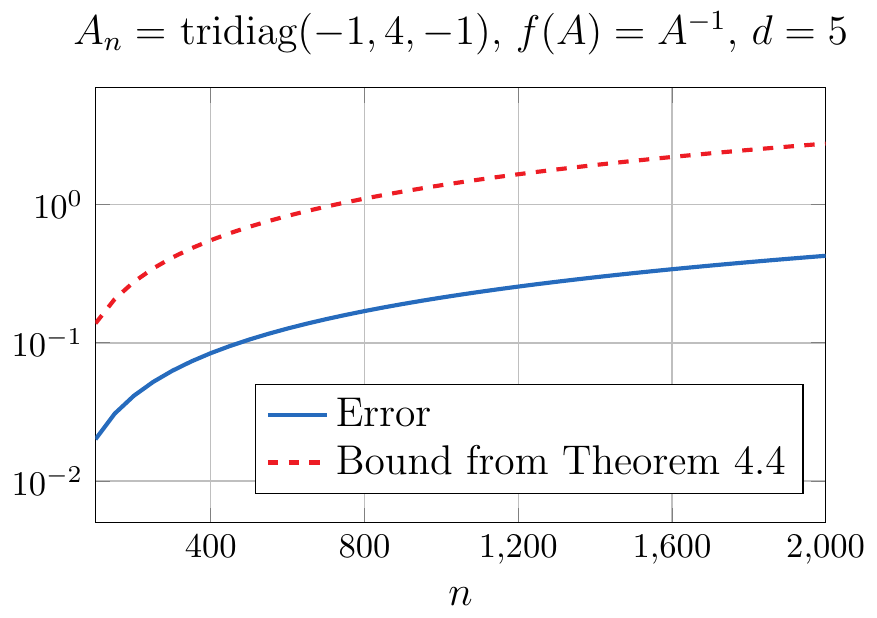}
\end{subfigure}
\begin{subfigure}{.49\textwidth}
	\includegraphics[width=\textwidth]{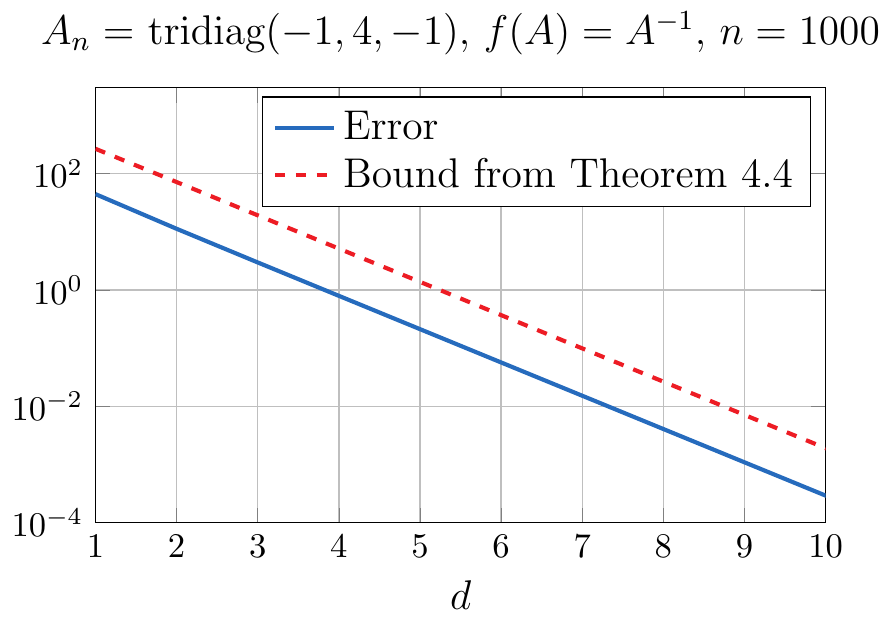}
\end{subfigure}
\begin{subfigure}{.49\textwidth}
	\includegraphics[width=\textwidth]{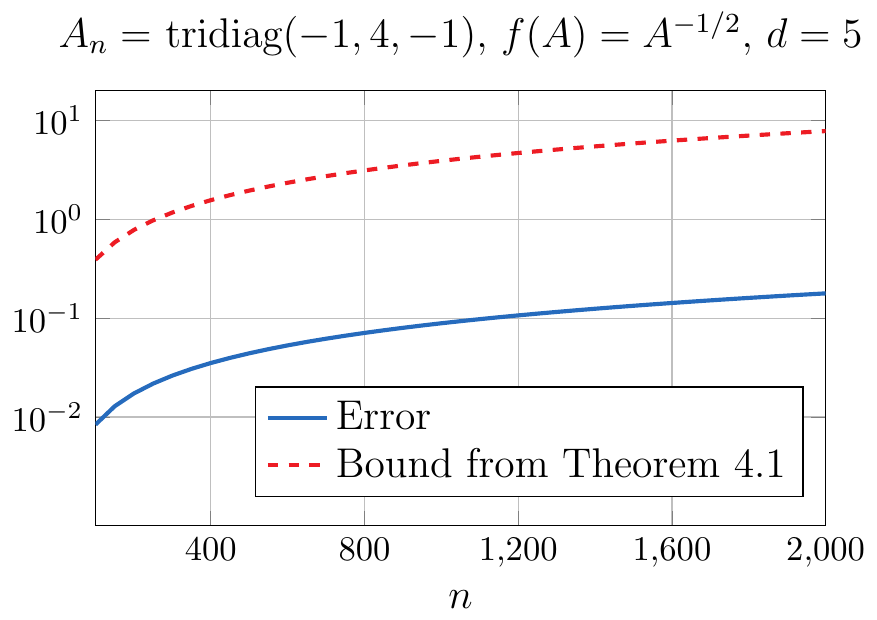}
\end{subfigure}
\begin{subfigure}{.49\textwidth}
	\includegraphics[width=\textwidth]{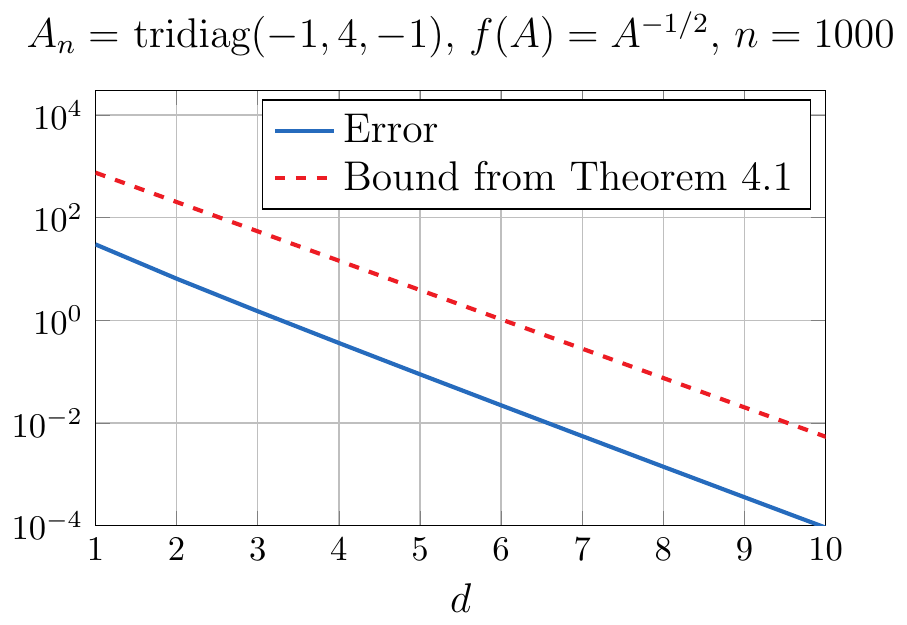}
\end{subfigure}
	\caption{Actual absolute error and error bound for the trace estimate~\eqref{eq:probing_trace_approximation} corresponding to the coloring~\eqref{eq:coloring_banded} for the matrix $A_n = \tridiag(-1,4,-1) \in \Cnn$ and $f(z) = 1/z$ (top row) and $f(z) = z^{-1/2}$ (bottom row). Results for varying $n$ are shown in the left column while results for varying $d$ are shown in the right column.\label{fig:tridiag_trace}}
\end{figure}

\review{Next, we turn to estimating} the trace for both matrix functions, using exactly the same experimental parameters as before and compare the actual error to the bound~\eqref{eq:errorbound2} from Theorem~\ref{the:boundsum}. Note that we could alternatively use the bound from Theorem~\ref{the:errorbound_banded} which is tailored to banded matrices. Both bounds almost agree here, the latter one being slightly less sharp, by a factor $\tfrac{1}{1-q^d}$. Figure~\ref{fig:tridiag_trace} shows that, as 
expected, the results are very similar to what can be observed in the context of computing a sparse approximation and we again observe a very good qualitative and quantitative agreement between the bounds and the actual error.

\begin{figure}
\begin{subfigure}{.49\textwidth}
		\includegraphics[width=\textwidth]{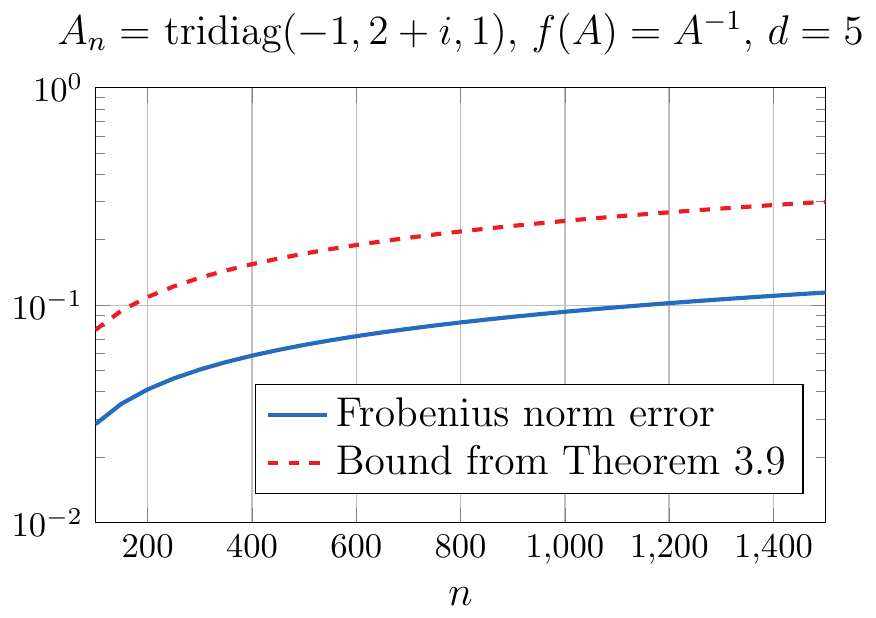}
\end{subfigure}
\begin{subfigure}{.49\textwidth}
	\includegraphics[width=\textwidth]{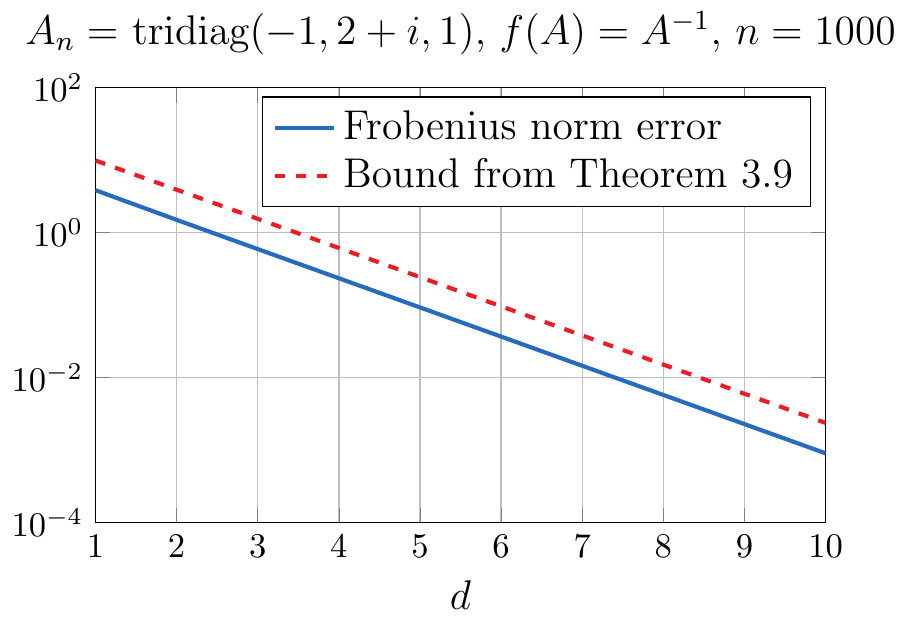}
\end{subfigure}
\caption{\review{Actual Frobenius norm error and error bound for the sparse approximation~\eqref{eq:probing_sparse_approximation} corresponding to the coloring~\eqref{eq:coloring_banded} for the shifted skew-Hermitian matrix $A_n = \tridiag(-1,2+i,1) \in \Cnn$ and $f(z) = 1/z$ for varying $n$ (left) and $d$ (right).}}\label{fig:skew_inv_sparse}
\end{figure}

\begin{figure}
\begin{subfigure}{.49\textwidth}
	\includegraphics[width=\textwidth]{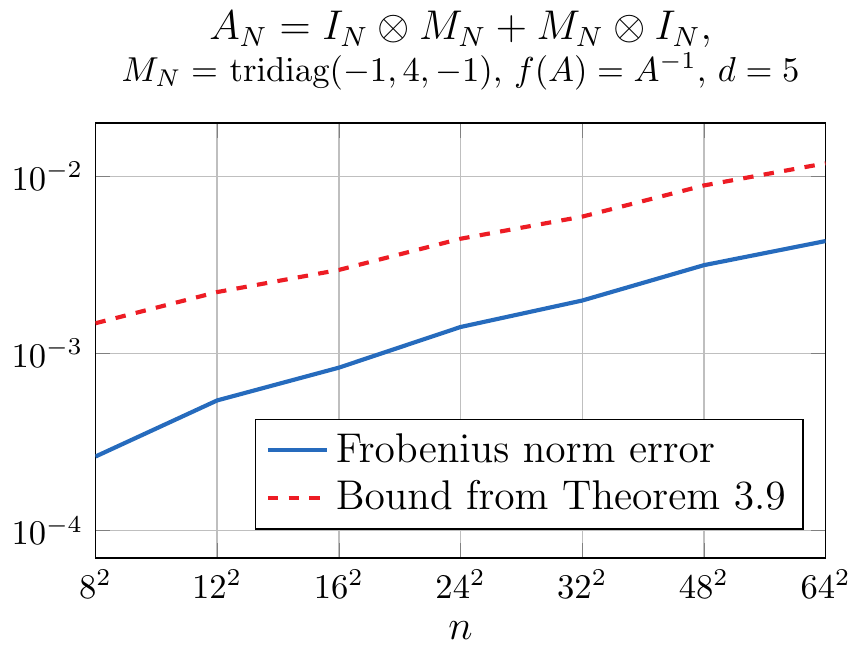}
\end{subfigure}
\begin{subfigure}{.49\textwidth}
	\includegraphics[width=\textwidth]{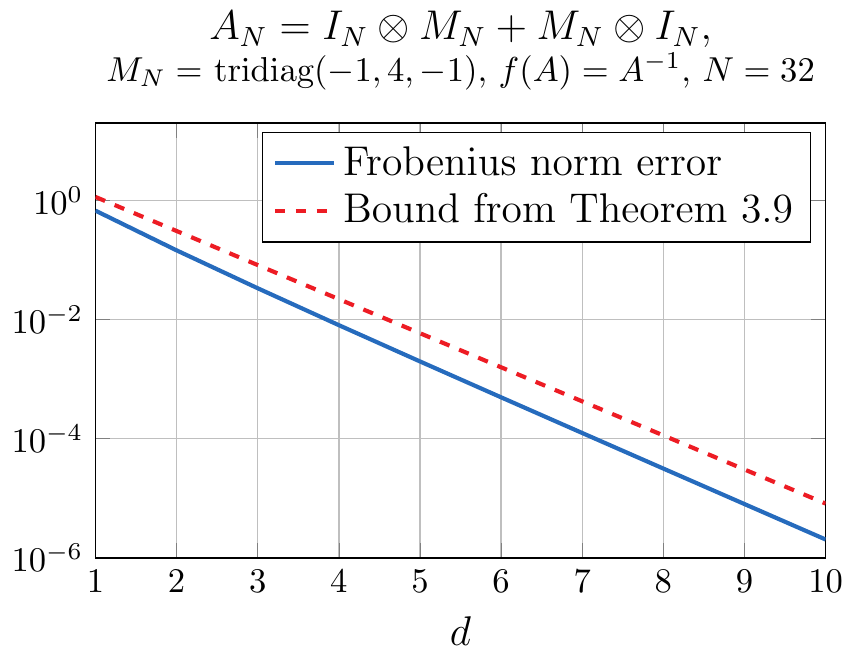}
\end{subfigure}
\caption{Actual Frobenius norm error and error bound for the sparse approximation~\eqref{eq:probing_sparse_approximation} corresponding to the coloring~\eqref{eq:coloring_lattice} for the matrix $A_N = I_N \otimes M_N + M_N \otimes I_N \in C^{N^2 \times N^2}$, where $M_N = \tridiag(-1,4,-1) \in \C^{N \times N}$ and $f(z) = 1/z$ for varying $N$ (left) and $d$ (right).
\label{fig:laplace_inv_sparse}}
\end{figure}

\review{To conclude this experiment, we demonstrate our results for functions of non-Hermitian matrices. We take the matrices $A_n = \tridiag(-1,2+i,1) \in \Cnn$, which is skew-Hermitian except for a  real-valued diagonal shift. Decay bounds for this class of matrices were given in~\cite[Theorem 3]{FrommerSchimmelSchweitzer2018_1}. We compute a coloring~\eqref{eq:coloring_banded} with $\beta = 1$ and use this to compute a sparse approximation of $A_n^{-1}$ for varying dimension $n$ while $d = 5$ is fixed, and for varying $d$ while $n = 1000$ is fixed. The results are given in Figure~\ref{fig:skew_inv_sparse}. As for the Hermitian case reported in Figure~\ref{fig:tridiag_inv_sparse}, the bounds accurately predict the scaling behavior of the actual error. In a similar manner, the results given in Figures~\ref{fig:tridiag_invsqrt_sparse}-\ref{fig:tridiag_trace} for the Hermitian case carry over to these non-Hermitian matrices, too, and we refrain from reporting them explicitly here.}

\subsection{Shifted two-dimensional Laplace operator} 

As a second model problem, we consider the family of matrices $A_N \in \C^{N^2 \times N^2}$ arising from discretization of the Laplace equation with homogeneous Dirichlet boundary conditions on a regular square grid\review{, with diagonal shifted by $4$, giving 
$$A_N = I_N \otimes M_N + M_N \otimes I_N \in C^{N^2 \times N^2},$$
where $M_N = \tridiag(-1,4,-1) \in \C^{N \times N}$  is the tridiagonal matrix from the previous experiment. Due to the shift, we obtain an $N$-independent decay in $f(A_N)$.} Applying a shift to the Laplacian matrix is common practice for obtaining \review{model problems with strong exponential decay}; see, e.g,~\cite{BenziSimoncini2015,Stathopoulos2013}, were the same (or similar) families of matrices were considered.

We have $\spec(A_N) \subset [4,12]$ independent of $N$ so that~\cite[Theorem 2.4]{DemkoMossSmith1984} guarantees an exponential decay of the entries of $A_N^{-1}$ with $C=\frac{1}{4}$ and $q = \tfrac{\sqrt{3}-1}{\sqrt{3+1}}$. We determine the color classes according to the optimal coloring for two-dimensional lattices from~\cite{FertinGodardRaspaud2003}. We again begin by approximating $A_N^{-1}$ for increasing values of $N$ while keeping $d = 5$ fixed and compare the actual error norm to the bound from Theorem~\ref{the:error_bound_sparse_approx_poly}. The results of this experiment are presented on the left-hand side of Figure~\ref{fig:laplace_inv_sparse} and we observe that the approximation error scales linearly with $N = \sqrt{n}$, as predicted by our theory. The magnitude of the error is overestimated by about one order of magnitude. On the right-hand side of Figure~\ref{fig:laplace_inv_sparse} the results for an experiment with varying $d$ and fixed $N = 32$ are given. Again, we observe good qualitative and quantitative agreement between the error bound and actual error norm.

\begin{figure}
\begin{subfigure}{.49\textwidth}
	\includegraphics[width=\textwidth]{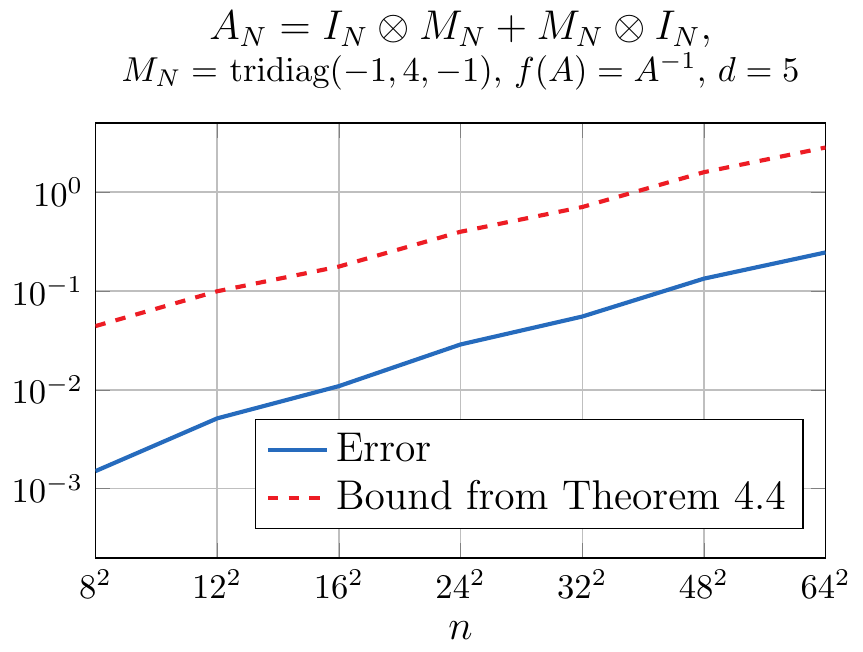}
\end{subfigure}
\begin{subfigure}{.49\textwidth}
	\includegraphics[width=\textwidth]{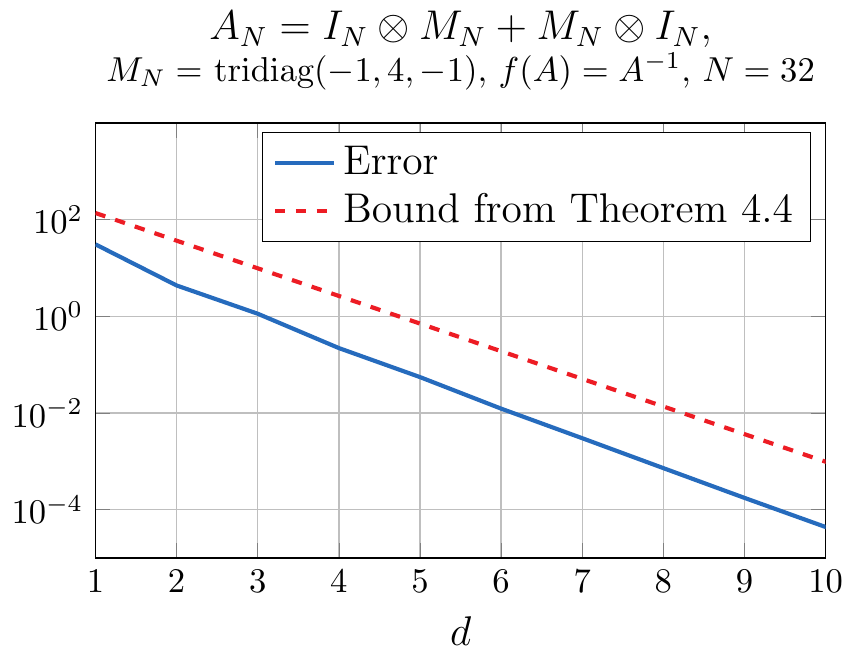}
\end{subfigure}
	\caption{Actual absolute error and error bound for the trace estimate~\eqref{eq:probing_trace_approximation} corresponding to the coloring~\eqref{eq:coloring_lattice} for the matrix $A_N = I_N \otimes M_N + M_N \otimes I_N \in C^{N^2 \times N^2}$, where $M_N = \tridiag(-1,4,-1) \in \C^{N \times N}$ and $f(z) = 1/z$ for varying $N$ (left) and $d$ (right).}\label{fig:laplace_inv_trace}
\end{figure}

\begin{figure}
\begin{subfigure}{.49\textwidth}
	\includegraphics[width=\textwidth]{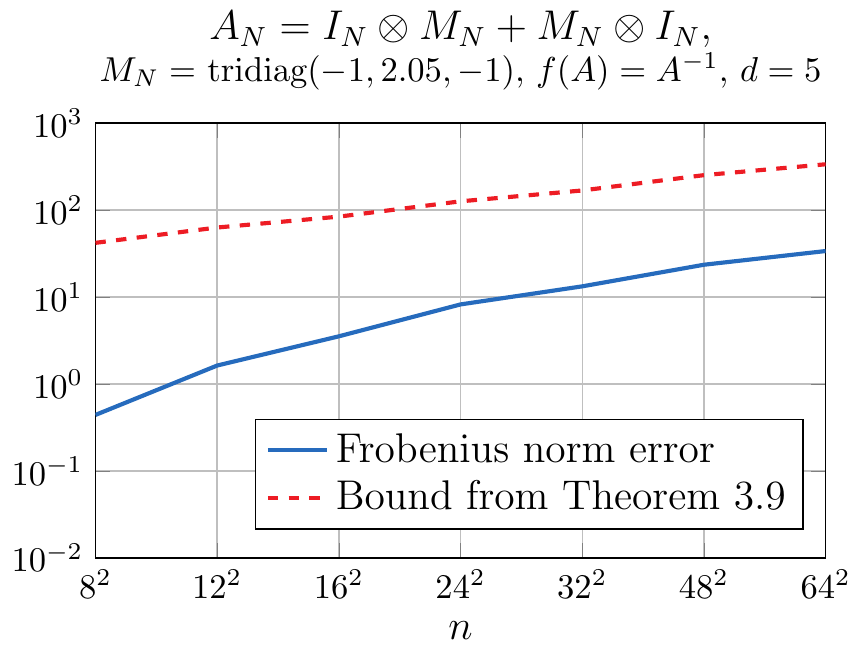}
\end{subfigure}
\begin{subfigure}{.49\textwidth}
	\includegraphics[width=\textwidth]{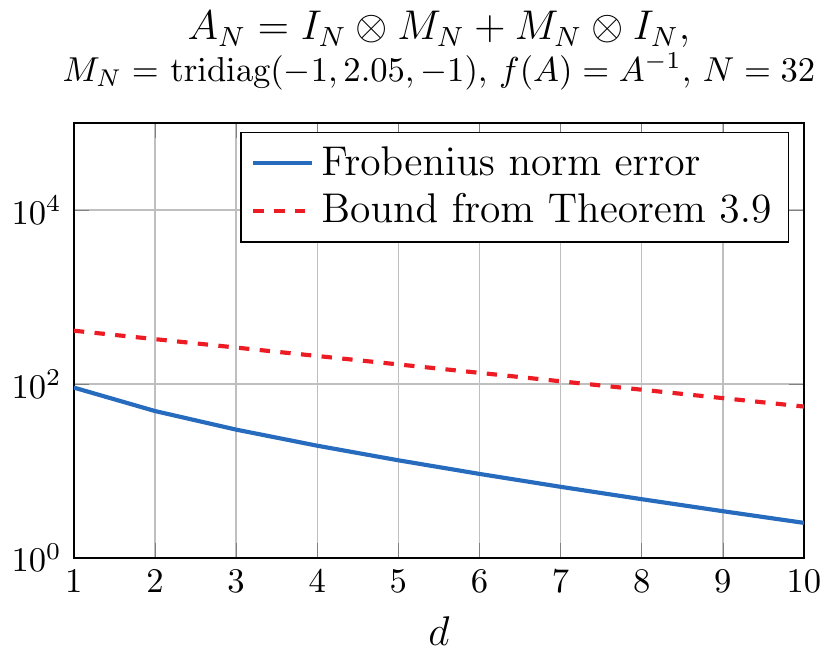}
\end{subfigure}
\caption{\review{Actual Frobenius norm error and error bound for the sparse approximation~\eqref{eq:probing_sparse_approximation} corresponding to the coloring~\eqref{eq:coloring_lattice} for the matrix $A_N = I_N \otimes M_N + M_N \otimes I_N \in C^{N^2 \times N^2}$, where $M_N = \tridiag(-1,2.05,-1) \in \C^{N \times N}$ and $f(z) = 1/z$ for varying $N$ (left) and $d$ (right).}}\label{fig:laplace_inv_sparse_ill}
\end{figure}

Next, in Figure~\ref{fig:laplace_inv_trace}, we also approximate $\tr(A_N^{-1})$, using the same experimental setup as for the sparse approximation and compare to the bound from Theorem~\ref{the:boundsum}. Again, we could also have used the lattice-specific bound from Theorem~\ref{the:errorbound_lattice}, which differs from that of Theorem~\ref{the:boundsum} by a factor $\tfrac{2}{(1-q^d)^2} \approx 2$ in this case. The results of this experiment are shown in Figure~\ref{fig:laplace_inv_trace}. The scaling behavior for growing $N$ and $d$ is again captured very accurately, although we overestimate the actual error norm by quite a large margin.

\review{
In order to investigate the influence of the conditioning of the matrix on the quality of our bounds, we conduct the first experiment again, this time shifting the discrete Laplace matrix by only $0.1$. This makes the matrices more ill-conditioned and leads to a much slower decay of the entries in $A^{-1}$. As can be seen from the results shown in Figure~\ref{fig:laplace_inv_sparse_ill}, this leads to much larger errors and slower error decrease in the probing method, which is, however, still captured quite well by our bounds. It is thus not primarily the quality of the bounds that is negatively influenced by the conditioning of $A$, but rather the performance of the probing method. 
}

\subsection{Thresholded covariance matrix}
For a next experiment, we consider the problem of computing a sparse approximation of an inverse covariance matrix, a task frequently occurring in uncertainty quanitification; see~\cite{BekasCurioniFedulova2009}. We use the  example matrix from~\cite{TangSaad2012}: Let $A_{N^2} = \cov(N,\alpha,\beta) \in \C^{N^2 \times N^2}$ be a covariance matrix corresponding to integer points $(x_i,y_i)$ arranged as a regular $N \times N$ grid with respect to a decaying, thresholded covariance function. More precisely, 
$$[A]_{ij} = \begin{cases} 
	\left(1 - \frac{\|(x_i,y_i) - (x_j,y_j)\|_2}{\alpha}\right)^\beta & \text{ if } \|(x_i,y_i) - (x_j,y_j)\|_2 \leq \alpha, \\
	0 & \text{ otherwise.}
\end{cases}$$

\begin{figure}
\begin{subfigure}{.49\textwidth}
	\includegraphics[width=\textwidth]{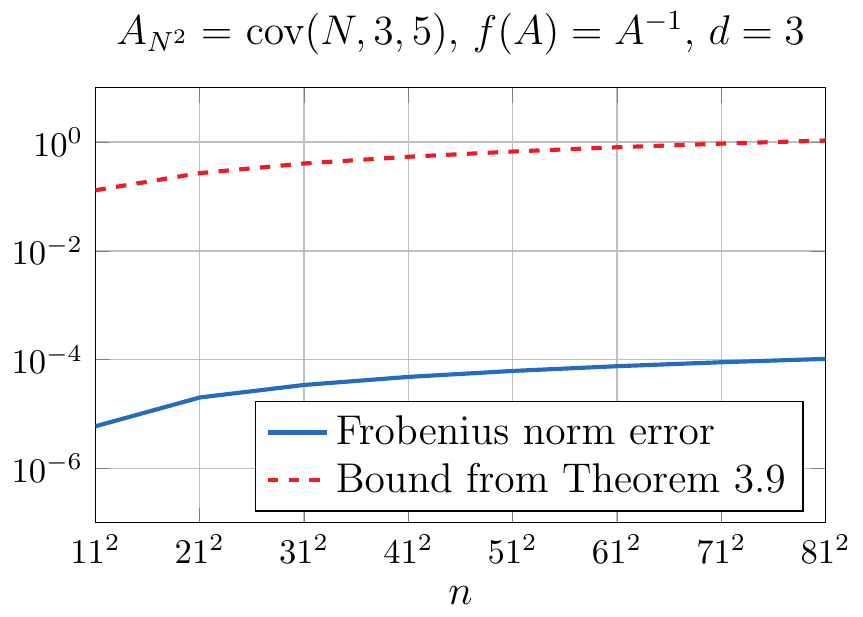}
\end{subfigure}
\begin{subfigure}{.49\textwidth}
	\includegraphics[width=\textwidth]{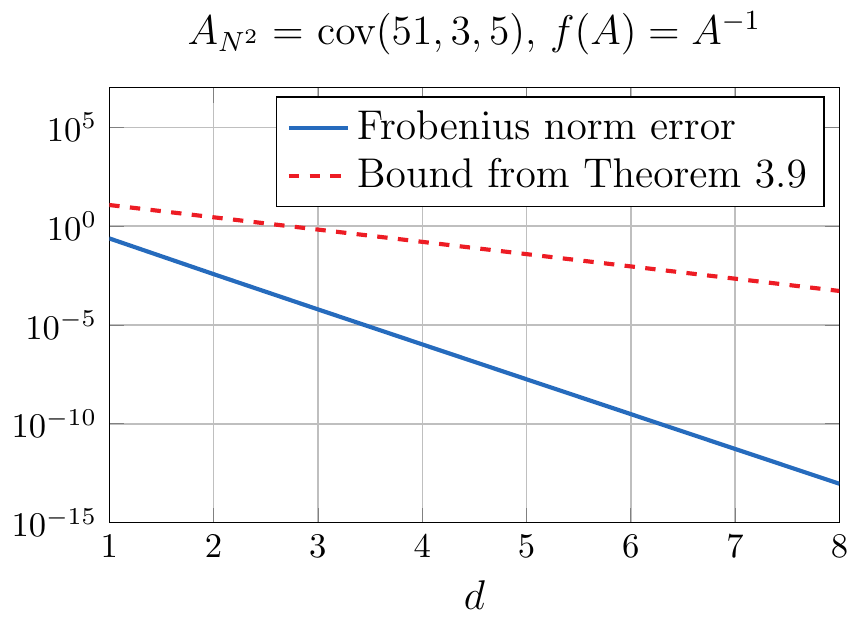}
\end{subfigure}

\begin{subfigure}{.49\textwidth}
	\includegraphics[width=\textwidth]{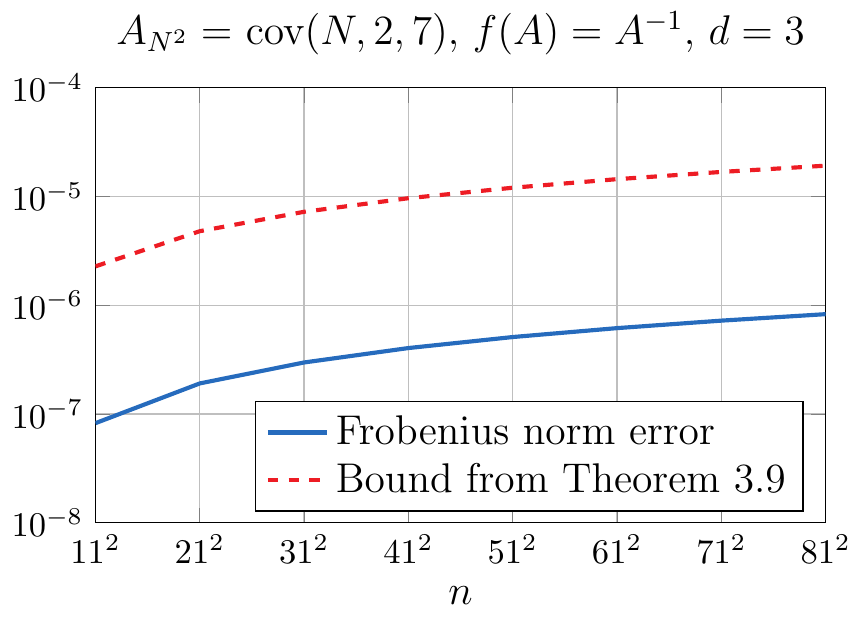}
\end{subfigure}
\begin{subfigure}{.49\textwidth}
	\includegraphics[width=\textwidth]{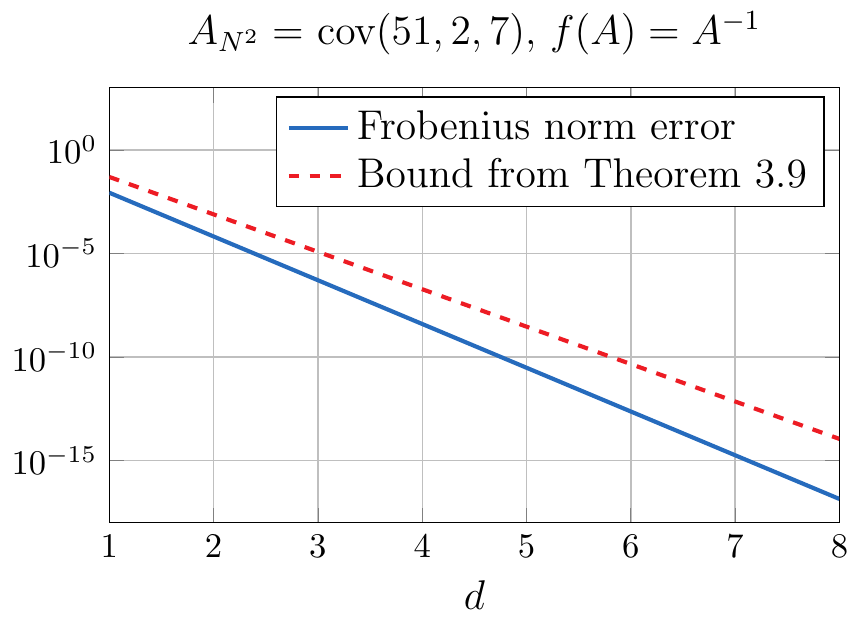}
\end{subfigure}
\caption{Actual Frobenius norm error and error bound for the sparse approximation~\eqref{eq:probing_sparse_approximation} corresponding to a greedy coloring for the matrix $A_{N^2} = \cov(N,\alpha,\beta) \in \C^{N^2 \times N^2}$ and $f(z) = 1/z$ for the parameter sets $\alpha = 3, \beta = 5$ (top) and $\alpha = 2, \beta = 7$ (bottom). Results for varying $n$ are shown in the left column while results for varying $d$ are shown in the right column.}\label{fig:cov_uq}
\end{figure}

We use the two parameter sets $\alpha = 3, \beta = 5$ and $\alpha = 2, \beta = 7$ and compute a sparse approximation for $A^{-1}$. \review{These test matrices are extremely well-conditioned. For the first parameter set, the spectral interval is approximately given by $[0.64, 1.72]$, leading to $C~\approx 2$ and $q\approx 0.24$, while for the second parameter set, the spectral interval is approximately given by $[0.96, 1.04]$, leading to $C~\approx 2$ and $q\approx 0.016$.}

Again, we perform one experiment in which we vary $n = N^2$ while $d = 3$ is fixed and one experiment in which we vary $d$ while $n=51^2$ is fixed. The resulting Frobenius norms of the error together with our bounds are given in Figure~\ref{fig:cov_uq}. For the first parameter set, $\alpha = 3, \beta = 5$, we observe that while the qualitative behavior for growing $n$ is accurately reproduced by our bound, we overestimate the actual error by several orders of magnitude. Thus, the bounds do give a valuable insight into the scaling behavior of the method but are not useful for judging whether the computed approximation is accurate enough for the application at hand. For growing $d$, we also observe that the slope of the error curve is much steeper than predicted by our bound, showing that also the qualitative behavior of the actual error is not accurately captured here. For the second parameter set, $\alpha = 2, \beta = 7$, our bounds look much better. For varying $n$, we still get an accurate impression of the qualitative scaling behavior while overestimating the error norm only 
by about one order of magnitude. 
For varying $d$, we still do not get a completely accurate reflection of the slope of the error curve, but the slopes agree much better than before.

\subsection{Maximum likelihood estimation for Gaussian Markov Random Fields}
In a last experiment, we consider the problem of maximum likelihood estimation for Gaussian Markov Random Fields (GMRFs). A GMRF is a multivariate joint Gaussian distribution defined with respect to some underlying graph, where each random variable corresponds to a node of the graph. The GMRF can be described by the positive definite and sparse precision matrix $A \in \Rnn$ (which is the inverse of the covariance matrix $\Sigma$ of the Gaussian distribution). Often, the precision matrix is parameterized by some unknown parameter $\phi$, i.e., $A = A(\phi)$ which can be estimated by a maximum likelihood estimator. Let $x\in \Rn$ be a sample from the Gaussian distribution. The log-likelihood of this sample is then given by the functional
\begin{equation}\label{eq:loglikelihood}
\log p(x\mid\phi) = \log\det(A(\phi)) - x^TA(\phi)x + G,
\end{equation}
where $G$ is a constant independent of $\phi$; see, e.g.,~\cite{Han2015}. The computationally demanding part in the evaluation of~\eqref{eq:loglikelihood} is the evaluation of the log-determinant. Due to the 
relation
$$\log\det(A(\phi)) = \tr\log(A(\phi)),$$
the log-determinant can be estimated by the probing approximation~\eqref{eq:probing_trace_approximation} applied to the matrix logarithm. 

\begin{figure}
\begin{subfigure}{.49\textwidth}
	\includegraphics[width=\textwidth]{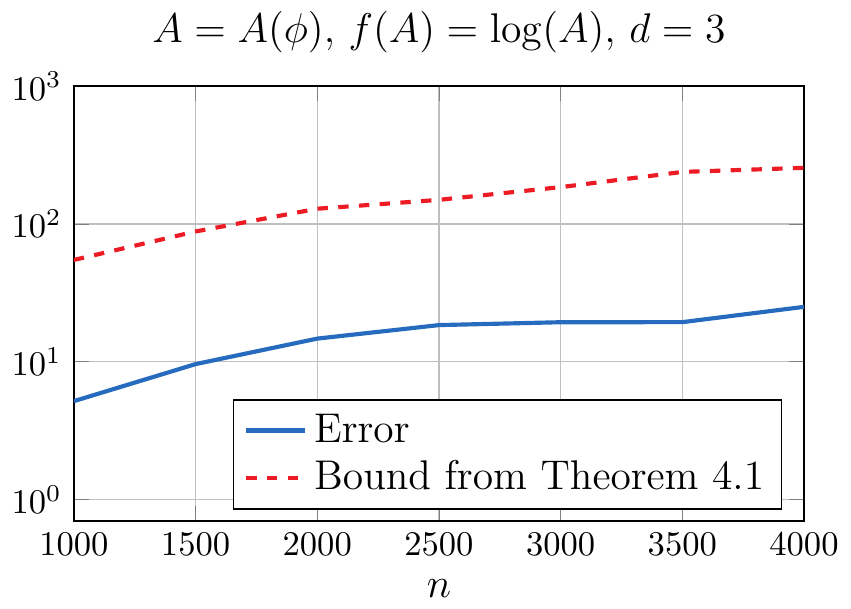}
\end{subfigure}
\begin{subfigure}{.49\textwidth}
	\includegraphics[width=\textwidth]{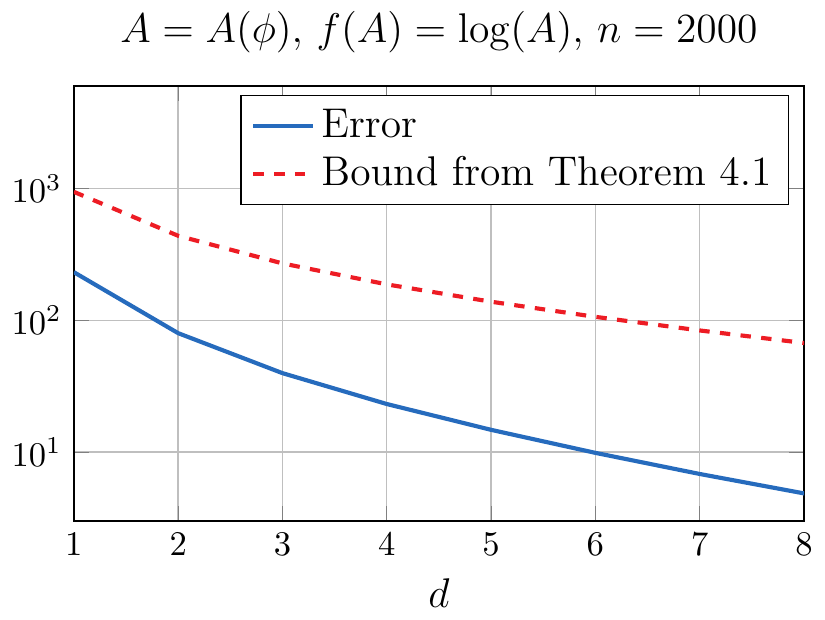}
\end{subfigure}
	\caption{Actual absolute error and error bound for the trace estimate~\eqref{eq:probing_trace_approximation} corresponding to the coloring~\eqref{eq:coloring_banded} for the precision matrix $A(\phi)$ of a GMRF and $f(z) = \log(z)$ for varying $n$ (left) and $d$ (right)}\label{fig:gmrf}
\end{figure}

We consider the GMRF model from~\cite{Pettitt2002}. Given a set of $n$ points $s_i \in [0, 1]$, we define a Gaussian random variable $x_i, i = 1,\dots,n$ at each point. The entries of the precision matrix are
\begin{equation}\label{eq:precision_matrix}
[A(\phi)]_{ij} =
\begin{cases}
1 + \phi\sum_{k = 1, k \neq i}^n \chi^\delta_{ij} & \text{ if } i = j, \\
-\phi \chi^\delta_{ij} & \text{ otherwise,}
\end{cases}
\end{equation}
where $\chi^\delta$ is given by
\begin{equation}\label{eq:distance_matrix}\nonumber
\chi^\delta_{ij} =
\begin{cases}
1 & \text{ if } \|s_i-s_j\|_2 < \delta, \\
0 & \text{ otherwise,}
\end{cases}
\end{equation}
where $\delta$ is a distance threshold which determines which points $s_i$ are connected in the graph underlying the GMRF. The resulting matrix is unstructured and sparse, but can be reordered to a matrix with rather narrow bandwidth by the Cuthill-McKee reordering, so that the coloring~\eqref{eq:coloring_banded} can be used. 

In our experiment, reported in Figure~\ref{fig:gmrf}, we fix $\phi = 20$ and use $\delta = 0.02$ when $n=1000$. For other values of $n$, we scale $\delta$ accordingly so that the average number of \review{nonzeros} per row and the bandwidth stay approximately constant for all values of $n$. In contrast to the previous experiments, we now mimick a situation that one typically faces in practice, namely that no explicit expressions for $C$ and $q$ in~\eqref{eq:generalDecaybound} are known, e.g., because the extremal eigenvalues of $A(\phi)$ are not known. In this case, one can obtain heuristic decay estimates by computing one (or a few) columns of $\log(A(\phi))$ (e.g., by a Krylov subspace method) and then estimating $C$ and $q$ from the observed decay pattern.

First, we vary $n$ between $1000$ and $5000$ while keeping $d=3$ fixed. Then, we fix $n=2000$ and vary $d$ between $1$ and $10$. \review{This results in matrices with spectral interval in $[1, 1300]$ and heuristically determined parameters of $C\approx 0.05$ and $q \approx 0.85$}. We compare the actual error of the probing approximation for the trace of the logarithm to the bound from Theorem~\ref{the:errorbound_banded} for banded matrices, using the estimated values of $C$ and $q$ computed from a single column of $\log (A(\phi))$. In both cases, we can observe a good qualitative agreement between our bound and the actual error.

\section{Conclusions}\label{sec:conclusion}
We have presented a detailed a priori error analysis of probing methods for the computation of sparse approximations and trace estimates of matrix functions, with a special emphasis on graph coloring based probing and matrix functions that exhibit an exponential decay. As illustrated in several numerical experiments, our error bounds accurately predict the scaling behavior of the error with respect to the matrix dimension $n$ or the coloring distance $d$. A particularly interesting observation in this context is that the error of the trace estimates decreases with exponent $d$, while the error of sparse approximations decreases only with exponent $\frac{d}{2}$. In addition to these error bounds for practical algorithms, we have also proven a new result on the existence of sparse approximations of matrix functions, improving on known results from the literature.
While our results typically give a 
good idea of the qualitative behavior of the actual error, they sometimes severely overestimate the actual error. 
Possible directions for future research include developing further ideas to improve the quality of the error bounds and looking at new approaches for efficient distance-$d$ coloring algorithms for appropriate classes of graphs.

\appendix

\section{Proof of Theorem~\ref{the:lattice_col} and Lemma~\ref{lemma:levelsets}}\label{sec:appendix}

\subsection{Proof of Theorem~\ref{the:lattice_col}}
Since for every node $w=(w^{[1]},\ldots,w^{[D]})$ we have $\widetilde{w^{[k]}} \in \{0,\ldots,d\}$ for $k=1,\ldots,D$, we know that the coloring
$$ \col(w)= \left (\sum\limits_{k=0}^{D-1} \widetilde{w^{[k]}} (d+1)^k \right ) + 1$$
produces at most $(d+1)^D$ colors. Now assume $\col(w)=\col(v)$ for nodes $w\neq v$. We want to show that $\dist(v,w)=\|v-w\|_1 > d$.
Because of 
$$\widetilde{w^{[k]}} = w^{[k]} \bmod (d+1)$$
we have $w^{[k]}=(d+1)a+\widetilde{w^{[k]}}$ and $v^{[k]}=(d+1)b+\widetilde{v^{[k]}}$ for some integers $a,b\geq 0$, and since $\col(w)=\col(v)$ we have $\widetilde{w^{[k]}}=\widetilde{v^{[k]}}$ for all $k=1,\ldots,D$. Since $w\neq v$ there exists at least one $k$ such that $w^{[k]}=(d+1)a+\widetilde{w^{[k]}} \neq  (d+1)b+\widetilde{v^{[k]}} = v^{[k]}$ which is equivalent to $a\neq b$ for $d\geq 0$. By fixing such a $k$ we obtain
$$\dist(w,v)=\|w-v\|_1 \geq |w^{[k]}-v^{[k]}| = (d+1)|a-b| \geq d+1$$
which proves the assertion. \cvd

\subsection{Proof of Lemma~\ref{lemma:levelsets}}
From \eqref{eq:size_of_distance_set_in_grid}
we obtain
 $$\sizeL_D^=(d)=\sizeL_D(d)-\sizeL_D(d-1)=\sum\limits_{k=0}^D \binom {D}{k} \binom{d+D-k-1}{D-1},$$
where we used $\binom{n+1}{k+1}=\binom{n}{k}+\binom{n}{k+1}.$

We will now use a proof technique called \emph{double counting} to prove that 
\begin{equation}\label{eq1}
\sum\limits_{k=0}^D \binom {D}{k} \binom{d+D-k-1}{D-1}
\end{equation}
is equal to
\begin{equation}\label{eq2}
 \sum\limits_{k=0}^{D-1} \binom{D}{k} \binom{d-1}{D-1-k} 2^{D-k}.
\end{equation}
For this, we first give a combinatorial interpretation of 
\eqref{eq1}, then formulate an equivalent statement which at last results in 
\eqref{eq2}.

Let $X=\{X_1,\ldots,X_D\}$ be a set with $D$ elements and let $Y=\{Y_1,\ldots,Y_{d-1}\}$ be a set with $d-1$ elements with $X\cap Y=\emptyset$. Then \eqref{eq1} counts the number of ways for choosing subsets $A\subseteq X$ and $B\subseteq X \cup Y$ with $|B|=D-1$ and $A\cap B=\emptyset$. This can be seen as 
follows: If $0\leq k\leq D$ is the number of elements in $A$, then $\binom{D}{k}$ counts the number of ways for choosing $A$.  Since $A\cap B=\emptyset$ there are $D+(d-1)-k$ elements left for the set $B$. Thus, the number of ways for choosing $B$ with $|B|=D-1$ is given by $\binom{d+D-k-1}{D-1}$. The sum over the number of elements in $A$ gives \eqref{eq1}.

Now, choosing such a $B\subseteq X \cup Y$ with $|B|=D-1$ and $A\cap B=\emptyset$ is equivalent to choosing subsets $N\subseteq X$ and $M\subseteq Y$ such that $|M|+|N|=D-1$ and $(N\cup M)\cap A = \emptyset$. Hence, we now count the number 
of ways for choosing subsets $A\subseteq X$, $N\subseteq X$ and $M\subseteq Y$ with $|M|+|N|=D-1$ and $(N\cup M)\cap A = \emptyset$. If $1\leq k \leq D-1$ is the number of elements in $M$, then there are $\binom{D}{k}$ ways for choosing $M$. The number of ways for choosing the left $D-1-k$ elements of $N$ out of $Y$ is given by $\binom{d-1}{D-1-k}$. Since $(N\cup M)\cap A = \emptyset$ there are $D-k$ elements left for $A$, i.e., there are $2^{D-k}$ ways for choosing $A$. The sum over the number of elements in $M$ gives \eqref{eq2}.

As a last step, we need to bound \eqref{eq2}, where we use 
$
\binom{n}{k} = 
\frac{n}{n-k}\binom{n-1}{k}$, 
$\binom {n}{k} \leq \frac{n^k}{k!}$ and $2^n\leq (n+1)!$. 
We then have 
\begin{align*}
  \sum\limits_{k=0}^{D-1} \binom{D}{k} \binom{d-1}{D-1-k} 2^{D-k} &= \sum\limits_{k=0}^{D-1} \frac{D}{D-k} \binom{D-1}{k} \binom{d-1}{D-1-k} 2^{D-k} \\
  &\leq D \sum\limits_{k=0}^{D-1} \frac 1{D-k} \binom{D-1}{k} \frac{(d-1)^{D-1-k}}{(D-1-k)!} 2^{D-k} \\
  &= 2D \sum\limits_{k=0}^{D-1} \binom{D-1}{k} \frac{(d-1)^{D-1-k}}{(D-k)!} 2^{D-k-1} \\
  &\leq 2D \sum\limits_{k=0}^{D-1} \binom{D-1}{k} (d-1)^{D-k-1}\\
  &=2Dd^{D-1},
\end{align*}
where the last equality comes 
from the binomial formula 
for $((d-1)+1)^{D-1}$. \cvd

We want to remark that the estimate from Lemma~\ref{lemma:levelsets} tends to severely overestimate the actual size of the level sets for larger values of $D$. 
This is exclusively due  
to the constant factor $2D$, while the factor $d^{D-1}$ is actually the sharpest one possible. 
Asymptotically, for fixed $D$ and growing $d$, we have
\[\sizeL_d^= \sim \frac{2^D}{(D-1)!}d^{D-1}.\]
This can be seen by carefully examining the proof above, noting that the asymptotic behavior is governed by the term corresponding to $k=0$, i.e., $\binom{d-1}{D-1} 2^{D}$ and 
that asymptotically, for growing $d$ we have 
$\binom {d-1}{D-1} \sim \frac{d^{D-1}}{(D-1)!}$.

\bibliographystyle{siam}
\bibliography{lit}

\end{document}